\documentclass[11pt]{article}
\usepackage[margin=1.0in]{geometry}
\usepackage{graphicx}
\usepackage{enumitem}
\usepackage{amsmath,amssymb,amsfonts,amsthm}
\usepackage{algorithm,algpseudocode}
\usepackage{array,longtable,booktabs}
\usepackage{hyperref}
\usepackage{fancyhdr}
\usepackage{threeparttable}
\usepackage{multirow}
\usepackage{makecell}

\usepackage{microtype}
\newtheorem{theorem}{Theorem}
\newtheorem{lemma}[theorem]{Lemma}
\newtheorem{proposition}[theorem]{Proposition}

\theoremstyle{definition}

\theoremstyle{remark}

\numberwithin{equation}{section}

\allowdisplaybreaks

\setlist[enumerate,1]{label=(\roman*),font=\normalfont}

\title{A path-following framework on fiber bundle for variational inequalities}

\author{
  Hongbo Sun\\
  \texttt{shb20@tsinghua.org.cn}
}

\begin{document}
\date{}
\maketitle

\begin{abstract}
  Developing general-purpose algorithms with global asymptotic convergence guarantees for variational inequalities (VIs) remains a longstanding challenge. This paper proposes a path-following framework for finite-dimensional VIs with arbitrary continuous mappings and polytope domains. The approach first transforms a general VI into a smooth VI on a simplex. Its key innovation is to formulate the resulting problem on a fiber bundle, called the fixed-point bundle, whose geometric structure enables the systematic integration of initialization, path-following, and singularity avoidance. Without imposing monotonicity or similar assumptions, the framework guarantees global linear convergence for nonsingular solutions and global linear convergence up to any prescribed precision for singular solutions. In numerical experiments, the algorithm attained a precision of $10^{-5}$ on all 6400 randomly generated instances, with dimensions up to 1024. Code and data are available at \url{https://github.com/shb20tsinghua/FiberBundle_VI}.
\end{abstract}



\tableofcontents

\section{Introduction}

Variational inequalities, introduced by Fichera \cite{vi_orig1} and Stampacchia \cite{vi_orig2} in the 1960s, provide a unifying framework for optimization \cite{vi_app_opt}, game theory \cite{vi_app_game}, economics \cite{vi_app_econ}, traffic assignment \cite{vi_app_traffic}, contact mechanics \cite{vi_app_contact}, fluid flow \cite{vi_app_flow}, machine learning \cite{vi_app_ml}, and many other fields \cite{vi_app}. A finite-dimensional variational inequality problem ${\rm VI}(H,K)$ is specified by a nonempty compact convex set $K\subseteq\mathbb{R}^n$ and a continuous map $H:K\to\mathbb{R}^n$, and the goal is to find $x^\star\in K$ such that
$$\left\langle H(x^\star),x'-x^\star\right\rangle \geq0 \quad \text{for all } x'\in K.$$
The variational inequality ${\rm VI}(H,K)$ is naturally described by the gap function
$${\rm gap}(x)=\sup_{x'\in K}\left\langle H(x),x-x'\right\rangle,$$
where ${\rm gap}(x)\geq0$ for all $x\in K$. Solutions are precisely the points $x^\star\in K$ with ${\rm gap}(x^\star)=0$. Existence of a solution follows from the Brouwer fixed-point theorem \cite{exist_vi_fp}. However, unconditional global convergence of numerical methods remains a challenge. This paper addresses general ${\rm VI}(H,K)$ with continuous $H$ and polytope $K$, where polytopes can approximate general compact convex sets by convex-body approximation.

In terms of iterative methods with asymptotic convergence, existing methods fall into several categories. Projection methods \cite{project1,project2} and proximal-point methods \cite{proximal1,proximal2} reformulate the VI as a fixed-point problem and iterate projections onto $K$. Merit-function approaches \cite{merit1,merit2} convert the VI into an optimization problem via merit functions such as gap functions and then apply gradient or Newton methods. Interior-point methods \cite{interior1,interior2} treat the VI as a complementarity problem and follow a central path parameterized by a barrier parameter. Homotopy methods \cite{homotopy1} deform a simple map into the target map. Operator splitting \cite{splitting1,splitting2,splitting3} handles additive decompositions. Modern variants add adaptive step sizes \cite{adaptive_proximal,adaptive_extragradient}, inertial terms \cite{inertial_split,inertial_proximal}, or stochastic estimates \cite{stochastic_proximal,stochastic_proximal,stochastic_variance} for large-scale problems.

Most existing algorithms require monotonicity, pseudo-monotonicity, or similar structural assumptions. Algorithms that handle non-monotone problems usually need additional conditions or provide no guarantees. In particular, homotopy methods guarantee the existence of a path from almost every starting point to a solution, but path existence does not yield a globally convergent path-following iteration, especially when the path has singular points. In short, in terms of iterative methods with asymptotic convergence, no existing method guarantees global convergence for every continuous $H$ on a polytope $K$ without additional assumptions.

Among these approaches, interior-point methods are particularly relevant because of their path-following philosophy. They trace a central path parameterized by a barrier parameter $\mu$, each of whose points satisfies the perturbed KKT conditions \cite{interior_point_original}. Our approach follows the same path-following spirit and is closely related to interior-point methods.

The proposed approach first preprocesses the general ${\rm VI}(H,K)$ into a smooth ${\rm VI}(F,\Delta)$ on a simplex $\Delta$. Every polytope is a linear image of a simplex, and every continuous map on a simplex is smoothed into a real-analytic function by averaging against a Dirichlet distribution. All subsequent development applies to the resulting smooth problem ${\rm VI}(F,\Delta)$.

We derive three equivalent characterizations of the paths that solve ${\rm VI}(F,\Delta)$, based on linear programming, Brouwer fixed points, and mixed complementarity problems. They yield the fixed-point KKT system (FPKKT), a simultaneous system of perturbed KKT conditions and a fixed-point equation. The paths satisfying FPKKT lead to solutions of ${\rm VI}(F,\Delta)$ and are special central paths, which means that our approach is equivalently a specialized interior-point method that follows these special central paths.

Rather than solve FPKKT directly, we study the geometry of its solution set, which naturally forms a fiber bundle over the simplex $\Delta$ with one-dimensional fibers in $\mu$. Two sections of this bundle recover the gap function of ${\rm VI}(F,\Delta)$ and govern the differential properties of the bundle. Analytic curves on the bundle connect a chosen starting point to solutions, and a parity argument implies that the number of nonsingular solutions is odd.

The fixed-point bundle is the central innovation. Its key insight is that the geometry separates motion of $\sigma\in\Delta$ along the base from motion of $\mu$ along the fibers. The starting point of the path lies on the fiber over a designated initial point $\sigma_{\rm init}\in\Delta$. Singular points on each fiber correspond to real eigenvalues of the matrix function $J(\sigma)$, and the bundle structure provides fiber jumps as a simple mechanism to bypass singular points while preserving the path structure.

The fixed-point bundle unifies the selection of a starting point, path-following, and singularity avoidance in a systematic way, yielding a path $\tilde{\Gamma}$ on the bundle. We then construct a compact nonsingular neighborhood $\mathcal{N}_{\tilde{\Gamma}}(\zeta)$ of this path on which the path-following algorithm is well-behaved; here $\zeta$ controls the prescribed precision at which path-following terminates before reaching the exact solution if the solution itself is singular.

The path-following algorithm follows the predictor-corrector framework. Using bounds established on $\mathcal{N}_{\tilde{\Gamma}}(\zeta)$, we prove uniform quadratic convergence of the corrector and a uniform bound on the predictor step length. The algorithm then drives ${\rm gap}(\sigma)$ to global linear convergence up to a prescribed precision $\zeta$. For nonsingular solutions $\zeta=0$, and the algorithm is globally linearly convergent; for singular solutions it retains global linear convergence up to $\zeta>0$.

We provide two different but closely related regularized Newton correctors for practical use. Numerical experiments on 6400 randomly generated instances of dimension up to 1024 show that the algorithm reduces the gap function below $10^{-5}$ in every case.

The remainder of the paper is organized as follows. Section 2 details the preprocessing. Section 3 presents the three equivalent characterizations. Section 4 constructs the fixed-point bundle and studies its differential properties. Section 5 formalizes the paths on the fixed-point bundle. Section 6 develops the predictor-corrector scheme and proves convergence. Section 7 presents experimental results. Concluding remarks appear in Section 8.

For vectors, $a\circ b$ denotes componentwise multiplication and $a/b$ componentwise division, and the inequality $a\geq b$ is also interpreted componentwise. The limit $\mu\to\mathbf{0}^+$ means $\mu\to\mathbf{0}$ with $\mu>0$. For a matrix $A$, componentwise multiplication means $A\circ b=A{\rm diag}(b)$ and $b\circ A={\rm diag}(b)A$. Unless otherwise specified, $\lVert a\rVert$ and $\lVert A\rVert$ denote the $L^2$ norm of a vector $a$ or a matrix $A$.

\section{Preprocessing: smooth VI over simplex}

The subsequent theory is developed for a variational inequality ${\rm VI}(F,\Delta)$ with a smooth mapping and a simplex domain. To place the framework in a sufficiently general setting, we first reduce a general variational inequality ${\rm VI}(H,K)$ with a continuous mapping on a polytope domain to ${\rm VI}(F,\Delta)$.

Let $V=(v_1,\dots,v_k)$ be the vertices of the polytope domain $K$, and represent any $x\in K$ as $x=V\sigma$ with $\sigma\in\Delta$. The variational inequality ${\rm VI}(H,K)$ is then transformed into finding $\sigma^\star\in\Delta$ such that, for all $\sigma\in\Delta$,
$$\left\langle H(x^\star),x-x^\star\right\rangle=H(V\sigma^\star)^\top(V\sigma-V\sigma^\star)=\left\langle V^\top H(V\sigma^\star),\sigma-\sigma^\star\right\rangle\geq0.$$
Thus ${\rm VI}(H,K)$ is transformed into a VI with simplex domain $\Delta$ and continuous mapping $F(\sigma)=V^\top H(V\sigma)$.

\subsection{Dirichlet smoothing}

Next we smooth the continuous mapping $F(\sigma)=V^\top H(V\sigma)$ on the simplex. General continuous functions can be approximated by smooth functions using kernel averaging, and since our domain is a simplex, we use the Dirichlet distribution as the kernel. The Dirichlet distribution ${\rm Dir}(\alpha)$ on the simplex $\Delta$ is given by
\begin{equation}
  {\rm Dir}(y;\alpha)=\frac{1}{B(\alpha)}\prod_{i=1}^n y_i^{\alpha_i-1},\quad\alpha>0,
\end{equation}
where $y\in\Delta$ is the variable, $\alpha>0$ is a vector of concentration parameters, and $B(\alpha)$ is the Beta function. For a continuous function $F:\Delta\to\mathbb{R}^n$, the smoothed function $F_\epsilon:\Delta\to\mathbb{R}^n$ is given componentwise by
\begin{equation}
  \label{equ_mollifi}
  F_{\epsilon,i}(\sigma)=\mathbb{E}_{Y\sim{\rm Dir}(\alpha)}\left[F_i(Y)\right]=\int_\Delta F_i(y){\rm Dir}(y;\sigma/\epsilon+c)\,dy.
\end{equation}
For the $i$-th component, $F_{\epsilon,i}(\sigma)$ is the expectation of $F_i(Y)$ under the Dirichlet distribution ${\rm Dir}(\alpha)$ with concentration parameter $\alpha=\sigma/\epsilon+c$. Here $\epsilon>0$ is a scalar that measures the approximation error, and $c>0$ is a constant vector ensuring that $\sigma/\epsilon+c>0$ for all $\sigma\in\Delta$.

The smoothing in equation~\eqref{equ_mollifi} is a mollification-type approach \cite{molli}. The Dirichlet distribution ${\rm Dir}(\sigma/\epsilon+c)$ acts as a mollifier: it is nonnegative on a compact domain, integrates to one, and concentrates its mass at $\sigma$ as $\epsilon\to0$. As in mollification, the smoothed function converges uniformly to the original function, satisfying
$$\lim_{\epsilon\to0} \sup_{\sigma\in\Delta} \left\lVert F_\epsilon(\sigma)-F(\sigma)\right\rVert=0.$$

Different parts of this paper require different degrees of smoothness. The computations in the resulting algorithm use the first derivative, the convergence of the algorithm is guaranteed by second-order continuous differentiability, and the compactness of the path is derived from real-analyticity. We therefore prove the strongest smoothness result, namely that $F_\epsilon(\sigma)$ is real-analytic.

\begin{proposition}
  \label{thm_analytic}
  For every $\epsilon>0$, $F_\epsilon(\sigma)$ is real-analytic on $\Delta$.
\end{proposition}
\begin{proof}
  A standard theorem on holomorphic parametric integrals in complex analysis can be used to prove this \cite{analytic}. Specifically, for a complex function $f(z,y)$ with $z\in Z$ and $y\in Y$, if $f(z,y)$ is holomorphic in $z$ and uniformly integrable in $y$, then the parametric integral $h(z)=\int_Y f(z,y) \,dy$ is holomorphic on $Z$. Uniform integrability means that there exists an integrable function $g(y)$ such that $\lvert f(z,y)\rvert<g(y)$ for every $z\in Z$.

  In our case, we study the complex parametric function $F_i(y){\rm Dir}(y;z/\epsilon+c)$ with $y\in\Delta$ and $z\in U_\sigma$. Here $U_\sigma=\{z\in\mathbb{C}^n|\lVert z-\sigma\rVert\leq d\}$ is a complex neighborhood of $\sigma\in\Delta$, and we require that ${\rm Re}(z/\epsilon+c)\geq m>0$ for every $z\in U_\sigma$. Such $d$ and $m$ exist for every $\sigma\in\Delta$ because $c>0$. We then show that $F_i(y){\rm Dir}(y;z/\epsilon+c)$ is holomorphic in $z$ and uniformly integrable in $y$.

  For holomorphy, the Beta function $B(\alpha)$ is known to be holomorphic and nonzero on $\{\alpha\in\mathbb{C}^n|{\rm Re}(\alpha)>0\}$. Hence $1/B(z/\epsilon+c)$ is holomorphic on $U_\sigma$ since ${\rm Re}(z/\epsilon+c)\geq m>0$. The term $\prod_{i=1}^n y_i^{z_i/\epsilon+c_i-1}$ is holomorphic in $z$ for every $y>0$ because it is an exponential function. Thus $F_i(y){\rm Dir}(y;z/\epsilon+c)$ is holomorphic in $z$ on $U_\sigma$ for every $y>0$. The set where $\min_i y_i=0$ has measure zero and does not affect the integral.

  For uniform integrability, $F_i(y)$ is continuous on the compact simplex $\Delta$, so it satisfies a bound $\lvert F_i(y)\rvert<M_F$, and $1/B(z/\epsilon+c)$ is continuous on the compact complex neighborhood $U_\sigma$, so it satisfies a bound $\lvert 1/B(z/\epsilon+c)\rvert<M_B$. Then
  $$\left\lvert F_i(y){\rm Dir}(y;z/\epsilon+c)\right\rvert\leq M_FM_B\left\lvert \prod_{i=1}^n y_i^{z_i/\epsilon+c_i-1}\right\rvert\leq M_FM_B\prod_{i=1}^n y_i^{{\rm Re}(z_i/\epsilon+c_i)-1}\leq M_FM_B\prod_{i=1}^n y_i^{m-1}.$$
  The function $\prod_{i=1}^n y_i^{m-1}$ is integrable on $\Delta$ as the Dirichlet multivariate Beta integral. Hence $F_i(y){\rm Dir}(y;z/\epsilon+c)$ is uniformly integrable in $y$ on $\Delta$.

  Therefore, by the standard holomorphic parametric integral theorem, $F_{\epsilon,i}(z)$ is holomorphic on the complex neighborhood $U_\sigma$ of every $\sigma\in\Delta$. It follows that $F_{\epsilon,i}(\sigma)$ is real-analytic in $\sigma\in\Delta$, and the vector function $F_\epsilon(\sigma)$ is real-analytic in $\sigma\in\Delta$ because each of its components is real-analytic.
\end{proof}

Besides $F_\epsilon(\sigma)$, the algorithm also relies on the derivative of $F_\epsilon(\sigma)$. For any distribution ${\rm D}(\alpha)$ and function $f(y)$, under some regularity conditions, there is a standard result called the score function identity \cite{score_identity}
\begin{equation*}
  \frac{\partial}{\partial\alpha}\mathbb{E}_{Y\sim{\rm D}(\alpha)}\left[f(Y)\right]={\rm Cov}_{Y\sim{\rm D}(\alpha)}\left[f(Y),S(Y;\alpha)\right],
\end{equation*}
where $S(Y;\alpha)=\partial\log{\rm D}(Y;\alpha)/\partial\alpha$ is the score function of the distribution ${\rm D}(\alpha)$ and ${\rm Cov}_{Y\sim{\rm D}(\alpha)}$ denotes covariance. For the Dirichlet distribution, $S_j(Y;\alpha)=\log Y_j-\psi(\alpha_j)+\psi(\mathbf{1}^\top\alpha)$, where $\psi$ is the digamma function. The covariance between a random variable and a constant is zero, so the derivative $F'_\epsilon(\sigma)$ is derived as
\begin{equation}
  \frac{\partial F_{\epsilon,i}(\sigma)}{\partial\sigma_j}=\frac{\partial\alpha_j}{\partial\sigma_j}\frac{\partial}{\partial\alpha_j}\mathbb{E}_{Y\sim{\rm Dir}(\alpha)}\left[F_i(Y)\right]=\frac{1}{\epsilon}{\rm Cov}_{Y\sim{\rm Dir}(\alpha)}\left[F_i(Y),\log Y_j\right].
\end{equation}

As $\epsilon\to0$, the derivative $F'_0(\sigma)$ is given by the following equation. Although the derivative does not converge to $F'(\sigma)$, the term $(I-\sigma\mathbf{1}^\top)$ can be absorbed by idempotence in the algorithm, which produces the same effect as $F'(\sigma)$.
\begin{equation}
  F'_0(\sigma)=\frac{\partial F_i(\sigma/(\mathbf{1}^\top\sigma))}{\partial\sigma_j}=F'(\sigma)\left(I-\sigma\mathbf{1}^\top\right)
\end{equation}

\subsection{Numerical evaluation via Monte Carlo sampling}

The function $F_\epsilon(\sigma)$ is given by the expectation of $F(Y)$ with $Y$ distributed as ${\rm Dir}(\sigma/\epsilon+c)$. Thus, its value can be evaluated by Monte Carlo sampling. Specifically, we draw $N$ independent samples $Y_k\sim{\rm Dir}(\sigma/\epsilon+c)$ and then estimate the expectation $F_{\epsilon,i}(\sigma)$ by the mean $\hat{F}_{\epsilon,i}(\sigma)$, as equation~\eqref{equ_mc_value} shows.
\begin{equation}
  \label{equ_mc_value}
  \hat{F}_{\epsilon,i}(\sigma)=\frac{1}{N}\sum_{k=1}^N F_i(Y_k)
\end{equation}

The estimation is unbiased, in the sense that $\mathbb{E}[\hat{F}_{\epsilon,i}(\sigma)]=F_{\epsilon,i}(\sigma)$. The estimation error is given by the variance derived as follows, where $\lvert F_i(y)\rvert<M_F$ is the bound that follows from continuity on the compact simplex.
\begin{equation}
  {\rm Var}\left[\hat{F}_{\epsilon,i}(\sigma)\right]=\frac{{\rm Var}[F_i(Y)]}{N}\leq\frac{\mathbb{E}[F_i(Y)^2]}{N}\leq\frac{M_F^2}{N}
\end{equation}

The derivative $F'_\epsilon(\sigma)$ is also evaluated by Monte Carlo sampling with $N$ independent samples $Y_k\sim{\rm Dir}(\sigma/\epsilon+c)$. Since ${\rm Cov}[F_i(Y),\log Y_j]=\mathbb{E}[F_i(Y)\log Y_j]-\mathbb{E}[F_i(Y)]\mathbb{E}[\log Y_j]$, the estimate $\hat{F}'_\epsilon(\sigma)$ of $F'_\epsilon(\sigma)$ is given as in equation~\eqref{equ_mc_value}.
\begin{equation}
  \label{equ_mc_derivative}
  [\hat{F}'_\epsilon(\sigma)]_{ij}=\frac{1}{\epsilon N}\sum_{k=1}^N F_i(Y_k)\left(\log Y_{k,j}-\mathbb{E}[\log Y_j]\right),~\mathbb{E}[\log Y_j]=\psi(\alpha_j)-\psi(\mathbf{1}^\top\alpha)
\end{equation}

The estimation is unbiased, in the sense that $\mathbb{E}[[\hat{F}'_\epsilon(\sigma)]_{ij}]=[F'_\epsilon(\sigma)]_{ij}$. The estimation error is given by the variance derived as follows, where $\psi_1$ is the trigamma function. In the derivation, apart from $\lvert F_i(y)\rvert<M_F$, we also use $\psi_1(\alpha_j)-\psi_1(\mathbf{1}^\top\alpha)<\psi_1(\alpha_j)\leq\psi_1(c_j)$, since $\psi_1$ is positive and monotonically decreasing on $(0,+\infty)$.
\begin{equation}
  \begin{aligned}
    {\rm Var}\left[[\hat{F}'_\epsilon(\sigma)]_{ij}\right]= & \frac{{\rm Var}[F_i(Y)(\log Y_j-\mathbb{E}[\log Y_j])]}{\epsilon^2N}\leq\frac{\mathbb{E}[F_i(Y)^2(\log Y_j-\mathbb{E}[\log Y_j])^2]}{\epsilon^2N}                      \\
    \leq                                                    & \frac{M_F^2{\rm Var}[\log Y_j]}{\epsilon^2N}=\frac{M_F^2}{\epsilon^2N}\left(\psi_1(\alpha_j)-\psi_1(\mathbf{1}^\top\alpha)\right)<\frac{M_F^2}{\epsilon^2N}\psi_1(c_j) \\
  \end{aligned}
\end{equation}

Therefore, to achieve a Monte Carlo estimation error $\epsilon_{\rm MC}$ and a Dirichlet smoothing error $\epsilon$ for both the function value $F_\epsilon(\sigma)$ and its derivative $F'_\epsilon(\sigma)$, the required number of samples $N$ satisfies $N=O(\epsilon_{\rm MC}^{-2}\epsilon^{-2})$. In other words, compared with evaluating an already smooth mapping, the computational cost of evaluating the Dirichlet-smoothed mapping increases polynomially in the Dirichlet error $1/\epsilon$ and the Monte Carlo error $1/\epsilon_{\rm MC}$.

\subsection{Approximation error to the original VI}

To conclude the preprocessing, we derive the approximation error of ${\rm VI}(F_\epsilon,\Delta)$ relative to ${\rm VI}(H,K)$ in terms of the gap function approximation error
$${\rm gap}_H(V\sigma)-{\rm gap}_{F_\epsilon}(\sigma)=\sup_{x'\in K}\left\langle H(V\sigma),V\sigma-x'\right\rangle-\sup_{\sigma'\in\Delta}\left\langle F_\epsilon(\sigma),\sigma-\sigma'\right\rangle.$$
Moreover, we derive the error assuming $K$ to be a general compact convex set, so as to incorporate the error of approximating convex bodies by polytopes. Specifically, Minkowski's theorem on inner approximation \cite{inner_approx} shows that for a finite-dimensional convex body $K$ and any required precision $\epsilon_{\rm set}$, there exists a polytope $S\subseteq K$ approximating $K$ to Hausdorff distance $d_H(K,S)=\max\left\{\sup_{x\in K}d(x,S),\sup_{s\in S}d(s,K)\right\}\leq\epsilon_{\rm set}$. The Dirichlet smoothing error is denoted by $\epsilon_{\rm map}$.

Let $x^{\prime\star}=\arg\max_{x'\in K}\langle H(V\sigma),V\sigma-x'\rangle$. By the inner approximation, there exists $\sigma^{\prime\prime\star}\in\Delta$ such that $\lVert V\sigma^{\prime\prime\star}-x^{\prime\star}\rVert<\epsilon_{\rm set}$. Hence
\begin{align*}
       & {\rm gap}_H(V\sigma)-{\rm gap}_{F_\epsilon}(\sigma)\leq\left\langle H(V\sigma),V\sigma-x^{\prime\star}\right\rangle-\left\langle F_\epsilon(\sigma),\sigma-\sigma^{\prime\prime\star}\right\rangle                                                 \\
  =    & \left\langle H(V\sigma),V\sigma^{\prime\prime\star}-x^{\prime\star}\right\rangle+\left\langle V^\top H(V\sigma)-F_\epsilon(\sigma),\sigma-\sigma^{\prime\prime\star}\right\rangle                                                                  \\
  \leq & \left\lVert H(V\sigma)\right\rVert \left\lVert V\sigma^{\prime\prime\star}-x^{\prime\star}\right\rVert+2\left\lVert V^\top H(V\sigma)-F_\epsilon(\sigma)\right\rVert\leq\left\lVert H(V\sigma)\right\rVert \epsilon_{\rm set}+2\epsilon_{\rm map}.
\end{align*}

Conversely, let $\sigma^{\prime\star}=\arg\max_{\sigma'\in \Delta}\langle F_\epsilon(\sigma),\sigma-\sigma'\rangle$. By the inner approximation, there exists $V\sigma^{\prime\star}\in K$. Hence
\begin{align*}
    & {\rm gap}_H(V\sigma)-{\rm gap}_{F_\epsilon}(\sigma)\geq\left\langle H(V\sigma),V\sigma-V\sigma^{\prime\star}\right\rangle-\left\langle F_\epsilon(\sigma),\sigma-\sigma^{\prime\star}\right\rangle \\
  = & \left\langle V^\top H(V\sigma)-F_\epsilon(\sigma),\sigma-\sigma^{\prime\star}\right\rangle\geq-2\left\lVert V^\top H(V\sigma)-F_\epsilon(\sigma)\right\rVert\geq-2\epsilon_{\rm map}.
\end{align*}

Thus, the gap function approximation error is given by equation~\eqref{equ_gap_error}. The coefficient $\left\lVert H(V\sigma)\right\rVert$ is bounded by continuity on $\Delta$. Therefore the approximation error between the gaps of ${\rm VI}(F_\epsilon,\Delta)$ and ${\rm VI}(H,K)$ is bounded by the set approximation error $\epsilon_{\rm set}$ and the function approximation error $\epsilon_{\rm map}$.
\begin{equation}
  \label{equ_gap_error}
  -2\epsilon_{\rm map}\leq{\rm gap}_H(V\sigma)-{\rm gap}_{F_\epsilon}(\sigma)\leq\left\lVert H(V\sigma)\right\rVert \epsilon_{\rm set}+2\epsilon_{\rm map}
\end{equation}

\section{Three equivalent characterizations regarding the paths}

The remainder of the paper studies ${\rm VI}(F,\Delta)$ with a smooth function $F$ and a simplex domain $\Delta$. The preprocessing is necessary because ${\rm VI}(F,\Delta)$ yields a particular system of equations called the fixed-point KKT system (FPKKT), which is the equation of the path we intend to follow. This section develops three equivalent characterizations of FPKKT, and hence of the paths, each of which reveals a different aspect of the framework. The linear programming form derives FPKKT from ${\rm VI}(F,\Delta)$, the Brouwer formulation establishes the existence of solutions of FPKKT, and the MCP formulation connects FPKKT to the central path in interior-point methods.

\subsection{Linear programming form and path equation}

To apply path-following methods, we seek an equation whose solutions coincide with those of ${\rm VI}(F,\Delta)$. We begin by introducing an optimization problem associated with ${\rm VI}(F,\Delta)$. Recall that ${\rm VI}(F,\Delta)$ equivalently requires $\sigma^\star\in\Delta$ such that $\sigma^{\star\top} F(\sigma^\star)\leq\sigma^{\prime\top} F(\sigma^\star)$ for all $\sigma'\in\Delta$. Equation~\eqref{equ_lp} shows a parameterized linear programming problem (LP) that resembles this requirement, where $\hat{\sigma}$ is the optimization variable and $\sigma$ is the parameter.
\begin{equation}
  \label{equ_lp}
  \begin{aligned}
    \min_{\hat{\sigma}} \quad & \hat{\sigma}^\top F(\sigma)   \\
    \textrm{s.t.}\quad        & \mathbf{1}^\top\hat{\sigma}=1 \\
                              & \hat{\sigma}\geq0             \\
  \end{aligned}
\end{equation}
The Lagrangian function of LP~\eqref{equ_lp} is
$$L=\hat{\sigma}^\top F(\sigma)+v(\mathbf{1}^\top\hat{\sigma}-1)-r^\top\hat{\sigma}=\hat{\sigma}^\top(F(\sigma)+v\mathbf{1}-r)-v,$$
where $r\geq0$ and $v$ are Lagrange multipliers. The $\mu=0$ case of equation~\eqref{equ_kkt}, together with $r\geq0$, constitutes the KKT conditions of LP~\eqref{equ_lp}. We additionally require the unbiased condition $\hat{\sigma}=\sigma$ to capture the requirement of ${\rm VI}(F,\Delta)$, and refer to the simultaneous equation~\eqref{equ_fpkkt} as the fixed-point KKT system (FPKKT).
\begin{subequations}
  \label{equ_fpkkt}
  \begin{equation}
    \label{equ_kkt}
    \begin{bmatrix}
      \hat{\sigma}\circ r-\mu       \\
      r-F(\sigma)-v\mathbf{1}       \\
      \mathbf{1}^\top\hat{\sigma}-1 \\
    \end{bmatrix}=0
  \end{equation}
  \begin{equation}
    \hat{\sigma}=\sigma
  \end{equation}
\end{subequations}

The $\mu\geq0$ case of equation~\eqref{equ_kkt} gives the perturbed KKT conditions of LP~\eqref{equ_lp}, where $\mu$ is a vector. Unlike the standard KKT conditions, we do not require $r\geq0$ explicitly, because $\mu\geq0$ already implies it if $\hat{\sigma}\in\Delta$.

LP~\eqref{equ_lp} and FPKKT~\eqref{equ_fpkkt} equivalently characterize ${\rm VI}(F,\Delta)$, as Theorem~\ref{thm_fpkkt} shows. This theorem establishes the equivalence between solutions of FPKKT~\eqref{equ_fpkkt} and ${\rm VI}(F,\Delta)$ for $\mu=0$. If FPKKT~\eqref{equ_fpkkt} has solutions at every $\mu>0$, these solutions form a path leading to solutions of ${\rm VI}(F,\Delta)$, analogous to the central path.
\begin{theorem}
  \label{thm_fpkkt}
  For $\sigma^\star\in\Delta$, the following statements are equivalent.
  \begin{enumerate}
    \item $\sigma^\star$ is a solution of ${\rm VI}(F,\Delta)$.
    \item $\sigma^\star$ is an optimizer of LP~\eqref{equ_lp} parameterized by $\sigma^\star$.
    \item $\sigma^\star$ satisfies the $\mu=0$ case of FPKKT~\eqref{equ_fpkkt} for some $v$ and $r\geq0$.
  \end{enumerate}
\end{theorem}
\begin{proof}
  ${\rm (i)}\Leftrightarrow{\rm (ii)}$:
  By definition, (i) is equivalent to $\langle F(\sigma^\star),\sigma'-\sigma^\star\rangle \geq0$ for all $\sigma'\in\Delta$, and (ii) is equivalent to $\sigma^{\star\top} F(\sigma^\star)\leq\sigma^{\prime\top} F(\sigma^\star)$ for all $\sigma'\in\Delta$. The equivalence follows.

  ${\rm (ii)}\Leftrightarrow{\rm (iii)}$:
  For every parameter $\sigma\in\Delta$, LP~\eqref{equ_lp} is a linear programming problem. A point $\hat{\sigma}$ is an optimizer of LP~\eqref{equ_lp} if and only if it satisfies the KKT conditions, which are given by equation~\eqref{equ_kkt} and $r\geq0$. Substituting $\hat{\sigma}=\sigma=\sigma^\star$ yields the equivalence.
\end{proof}

\subsection{Brouwer fixed-point form and path existence}

Having established the equivalence for $\mu=0$, we next address the existence of paths as $\mu\to\mathbf{0}^+$. For path-following, it is essential that FPKKT~\eqref{equ_fpkkt} admit solutions for every $\mu>0$. We therefore study the $\mu>0$ case.

The correspondence from $(\sigma,\mu)$ to $(\hat{\sigma},r,v)$ satisfying the perturbed KKT conditions~\eqref{equ_kkt} is called the Brouwer function
\begin{equation}
  (\hat{\sigma},r,v)=M(\sigma,\mu),\quad{\rm s.t.~perturbed~KKT~conditions~\eqref{equ_kkt}},(\sigma,\mu)\in\Delta\times\{\mu|\mu>0\}.
\end{equation}
We have the following proposition.
\begin{proposition}
  \label{thm_brouwer_continuous}
  $(\hat{\sigma},r,v)=M(\sigma,\mu)$ is a continuous function on $\Delta\times\{\mu|\mu>0\}$.
\end{proposition}
\begin{proof}
  The correspondence is single-valued. For every $(\sigma,\mu)$, the corresponding $v$ is the unique root of $q(v)=0$, where
  $$q(v)=\mathbf{1}^\top\frac{\mu}{F(\sigma)+v\mathbf{1}}-1,\quad \frac{dq(v)}{dv}=-\mathbf{1}^\top\frac{\mu}{\left(F(\sigma)+v\mathbf{1}\right)^2}.$$
  Indeed, $dq(v)/dv<0$. The limiting behavior is
  $$\lim_{v\to \left(-\min_i F_i(\sigma)\right)^+}q(v)=+\infty \quad \text{and} \quad \lim_{v\to \infty}q(v)=-1.$$
  Thus $q(v)$ decreases monotonically from $+\infty$ to $-1$ on $(-\min_i F_i(\sigma),+\infty)$. Consequently, $v$ is unique for every $(\sigma,\mu)$, hence so is $(\hat{\sigma},r,v)$.

  Moreover, since $q(-\min_i F_i(\sigma)+\mathbf{1}^\top\mu)\leq0$, the unique solution $v$ belongs to the interval $[-\min_i F_i(\sigma),-\min_i F_i(\sigma)+\mathbf{1}^\top\mu]$.

  Let $(\sigma_k,\mu_k)\to(\sigma_0,\mu_0)$ with $(\sigma_k,\mu_k)\in\Delta\times\{\mu|\mu>0\}$ and $(\sigma_0,\mu_0)\in\Delta\times\{\mu|\mu>0\}$, and let $v_k$ be the unique solution of $q(v)=0$ for each $k$. Since each $v_k$ lies in the bounded interval above, the sequence $\{v_k\}$ has a convergent subsequence with limit point $v_0$. By continuity of $F(\sigma)$ and of $q(\sigma,\mu,v)$ in $(\sigma,\mu,v)$, we have
  $$\mathbf{1}^\top\frac{\mu_0}{F(\sigma_0)+v_0\mathbf{1}}-1=0.$$
  Thus $v_0$ is the unique solution of $q(v)=0$ for $(\sigma_0,\mu_0)$. Since every convergent subsequence of $\{v_k\}$ converges to this same limit $v_0$, the entire sequence $\{v_k\}$ converges to $v_0$. Hence $v$ is continuous with respect to $(\sigma,\mu)$ on $\Delta\times\{\mu|\mu>0\}$, and continuity of $r=F(\sigma)+v\mathbf{1}$ and $\hat{\sigma}=\mu/r$ then yields continuity of $(\hat{\sigma},r,v)=M(\sigma,\mu)$.
\end{proof}

By the Brouwer fixed-point theorem, for every $\mu>0$ there exists a fixed point $\sigma=\hat{\sigma}$ of the Brouwer function $(\hat{\sigma},r,v)=M(\sigma,\mu)$. Fixed points of $(\hat{\sigma},r,v)=M(\sigma,\mu)$ are equivalent to solutions of FPKKT~\eqref{equ_fpkkt} by definition of the Brouwer function. Theorem~\ref{thm_exist} therefore guarantees the existence of paths subject to FPKKT~\eqref{equ_fpkkt} as $\mu\to\mathbf{0}^+$.
\begin{theorem}[Existence theorem]
  \label{thm_exist}
  For any $\mu>0$, there exists $\sigma\in\Delta$ satisfying FPKKT~\eqref{equ_fpkkt}.
\end{theorem}

The proof of Theorem~\ref{thm_exist} also suggests that the Brouwer function can be computed by solving the bisection problem~\eqref{equ_bisection}.
\begin{equation}
  \label{equ_bisection}
  q(v)=\mathbf{1}^\top\frac{\mu}{F(\sigma)+v\mathbf{1}}-1,\quad v\in[-\min_i F_i(\sigma),-\min_i F_i(\sigma)+\mathbf{1}^\top\mu]
\end{equation}

\subsection{Mixed complementarity problem form and central path}

Having established the existence of paths satisfying FPKKT~\eqref{equ_fpkkt}, we now show that FPKKT~\eqref{equ_fpkkt} can be interpreted as the equation of a distinguished central path of the mixed complementarity problem (MCP) in equation~\eqref{equ_mcp}. Deep connections between complementarity problems, variational inequalities, and Brouwer fixed-point problems are well established \cite{red_vi_fp,dual_gap,bepm_game}.
\begin{equation}
  \label{equ_mcp}
  \begin{aligned}
    \min_{(\sigma,r,v)} \quad & \sigma^\top r-\mu^\top\ln\sigma-\mu^\top \ln r \\
    \textrm{s.t.}\quad        & r=F(\sigma)+v\mathbf{1}                        \\
                              & \mathbf{1}^\top\sigma=1                        \\
                              & (\sigma,r)\geq0                                \\
  \end{aligned}
\end{equation}

The $\mu=0$ case of equation~\eqref{equ_mcp} is obtained by taking the complementarity pair $\sigma\circ r$ from FPKKT~\eqref{equ_fpkkt} as the objective, and the $\mu>0$ case is the barrier problem of the MCP. The Lagrangian function of MCP~\eqref{equ_mcp} is
$$L=\sigma^\top r+\bar{\lambda}^\top(r-F(\sigma)-v\mathbf{1})+\tilde{\lambda}(\mathbf{1}^\top\sigma-1)-\check{r}^\top\sigma-\check{\sigma}^\top r,$$
where $\bar{\lambda}$, $\tilde{\lambda}$, $\check{\sigma}\geq0$, and $\check{r}\geq0$ are Lagrange multipliers. The perturbed KKT conditions of MCP~\eqref{equ_mcp} are
\begin{equation}
  \label{equ_kkt_mcp}
  \begin{bmatrix}
    -\left(\partial F(\sigma)/\partial\sigma\right)^\top\bar{\lambda}+\tilde{\lambda}\mathbf{1}+r-\check{r} \\
    \bar{\lambda}+\sigma-\check{\sigma}                                                                     \\
    -\bar{\lambda}^\top\mathbf{1}                                                                           \\
    r\circ \check{\sigma}-\mu                                                                               \\
    \sigma\circ \check{r}-\mu                                                                               \\
    r-F(\sigma)-v\mathbf{1}                                                                                 \\
    \mathbf{1}^\top\sigma-1
  \end{bmatrix}=0.
\end{equation}

Theorem~\ref{thm_limit} shows that the three characterizations are equivalent representations of the same underlying object.
\begin{theorem}
  \label{thm_limit}
  For $(\sigma,\mu)\in\Delta\times\{\mu|\mu>0\}$, the following statements are equivalent.
  \begin{enumerate}
    \item $(\sigma,\mu)$ satisfies FPKKT~\eqref{equ_fpkkt} for some $r,v$.
    \item $\sigma$ is a fixed point such that $\sigma=\hat{\sigma}$ of the Brouwer function $(\hat{\sigma},r,v)=M(\sigma,\mu)$.
    \item $(\sigma,\mu)$ satisfies $\check{\sigma}=\sigma$ and the perturbed KKT conditions of MCP~\eqref{equ_mcp} for some $r,v,\bar{\lambda},\tilde{\lambda},\check{\sigma},\check{r}$.
  \end{enumerate}
\end{theorem}
\begin{proof}
  ${\rm (i)}\Leftrightarrow{\rm (ii)}$:
  The Brouwer function $(\hat{\sigma},r,v)=M(\sigma,\mu)$ is defined by the perturbed KKT conditions~\eqref{equ_kkt}. FPKKT~\eqref{equ_fpkkt} is the simultaneous system consisting of the perturbed KKT conditions~\eqref{equ_kkt} and $\sigma=\hat{\sigma}$, so the equivalence follows.

  ${\rm (i)}\Leftrightarrow{\rm (iii)}$:
  Given $(\sigma,\mu,r,v)$ satisfying FPKKT~\eqref{equ_fpkkt}, the tuple given by $(\sigma,\mu,r,v,\bar{\lambda},\tilde{\lambda},\check{\sigma},\check{r})=(\sigma,\mu,r,v,\mathbf{0},0,\sigma,r)$ satisfies $\check{\sigma}=\sigma$ and the perturbed KKT conditions~\eqref{equ_kkt_mcp}. Conversely, if $(\sigma,\mu,r,v,\bar{\lambda},\tilde{\lambda},\check{\sigma},\check{r})$ satisfies $\check{\sigma}=\sigma$ and the perturbed KKT conditions~\eqref{equ_kkt_mcp}, then the tuple $(\sigma,\mu,r,v)$ satisfies FPKKT~\eqref{equ_fpkkt}.
\end{proof}

Central paths are defined by perturbed KKT conditions. Thus the path given by FPKKT~\eqref{equ_fpkkt} is precisely the central path of MCP~\eqref{equ_mcp} on which $\check{\sigma}=\sigma$. As in interior-point methods, these paths lead to solutions of ${\rm VI}(F,\Delta)$ as $\mu\to\mathbf{0}^+$.

Note that $(\sigma,\mathbf{0})$ satisfying FPKKT~\eqref{equ_fpkkt} is not necessarily a solution of ${\rm VI}(F,\Delta)$, since the $\mu=0$ case additionally requires $r\geq0$. For $\sigma\in\Delta$, $r\geq0$ if and only if $\mu\geq0$, because $\sigma\circ r=\mu$. Hence a point $(\sigma^\star,\mathbf{0})$ satisfying FPKKT~\eqref{equ_fpkkt} must be a limit point as $\mu\to\mathbf{0}^+$ in order to be a solution of ${\rm VI}(F,\Delta)$. This is proved in the next section.

\section{Fiber bundle as the solution space of FPKKT}

The previous section established FPKKT~\eqref{equ_fpkkt} as an equivalent characterization of ${\rm VI}(F,\Delta)$ and characterized the paths subject to it as special central paths. We now introduce the central geometric object of the paper. Instead of viewing FPKKT~\eqref{equ_fpkkt} as an equation system, we regard its solution set as a geometric space. This viewpoint reveals a fiber-bundle structure that unifies starting point selection, path-following, singularity avoidance, and convergence to solutions of ${\rm VI}(F,\Delta)$. In this section we study the solution space of FPKKT~\eqref{equ_fpkkt} with $\mu\in\mathbb{R}^n$ and formally show that the paths as $\mu\to\mathbf{0}^+$ equivalently lead to solutions of ${\rm VI}(F,\Delta)$.

\subsection{The fixed-point bundle}

The solution space of FPKKT~\eqref{equ_fpkkt} has a natural fiber-bundle structure. For each $\sigma$, the admissible $\mu$ values form an affine one-dimensional family given by $\mu=\sigma\circ F(\sigma)+v\sigma$. The disjoint union of these affine lines over all $\sigma\in\Delta$ forms the fiber bundle of equation~\eqref{equ_fpbund}. The bundle is denoted $(E, \Delta, \alpha:E\to\Delta)$, where $E$ is the total space, $\Delta$ is the base space, $\alpha:E\to\Delta$ is the projection map, and $\alpha^{-1}(\sigma)=\{\sigma\}\times B(\sigma)$ is the fiber over $\sigma$.
\begin{equation}
  \label{equ_fpbund}
  \begin{aligned}
     & E=\bigcup_{\sigma\in\Delta}\{\sigma\}\times B(\sigma)                 \\
     & B(\sigma)=\left\{\sigma\circ F(\sigma)+v\sigma|v\in\mathbb{R}\right\} \\
     & \alpha\left((\sigma,\mu)\right)=\sigma                                \\
  \end{aligned}
\end{equation}

We refer to this as the fixed-point bundle, and by convention we also use the total space $E$ to refer to the fiber bundle. This geometric structure is not merely a change of terminology: it provides a characterization of the paths that separates motion of $\sigma$ along the base space from motion of $\mu$ along individual fibers, a separation that is crucial for starting point selection and singularity avoidance.

We eliminate $r$ and $v$ in FPKKT~\eqref{equ_fpkkt} to obtain a fixed-point bundle equation in $(\sigma,\mu)$ only. Applying the projection matrix $I-\sigma\mathbf{1}^\top$ to FPKKT~\eqref{equ_fpkkt} yields the equation $G(\sigma,\mu)=0$ of equation~\eqref{equ_g}. We work primarily with this equivalent formulation in place of FPKKT~\eqref{equ_fpkkt}; it describes the fixed-point bundle and serves as the fundamental equation throughout the remainder of the paper. We write $(\sigma,\mu)\in E$ to denote points satisfying $G(\sigma,\mu)=0$.
\begin{equation}
  \label{equ_g}
  G(\sigma,\mu)=\left(I-\sigma\mathbf{1}^\top\right)\left(\sigma\circ F(\sigma)-\mu\right),\quad(\sigma,\mu)\in\Delta\times\mathbb{R}^n
\end{equation}

The projection matrix $I-\sigma\mathbf{1}^\top$ appears throughout this paper, and the following properties are used frequently in the proofs and derivations. It projects along the vector $\sigma$, so that $(I-\sigma\mathbf{1}^\top)\sigma=0$; it projects onto the subspace orthogonal to $\mathbf{1}$, so that $\mathbf{1}^\top(I-\sigma\mathbf{1}^\top)=0$; and it is idempotent, so that $(I-\sigma\mathbf{1}^\top)(I-\sigma\mathbf{1}^\top)=I-\sigma\mathbf{1}^\top$. The linear equation $(I-\sigma\mathbf{1}^\top)x=b$ is solvable when $\mathbf{1}^\top b=0$, and its general solution is $x=b+k\sigma$ for $k\in\mathbb{R}$. Moreover, $I-\mathbf{1}\sigma^\top=(I-\sigma\mathbf{1}^\top)^\top$ is another projection matrix with similar properties.

We introduce two sections of the fixed-point bundle. A section of a fiber bundle is a map that sends every $\sigma$ to a point on the fiber $B(\sigma)$. The nonnegative section $\check{\mu}(\sigma)$ in equation~\eqref{equ_gap_section} satisfies $\check{\mu}(\sigma)\geq0$ and directly characterizes the gap of ${\rm VI}(F,\Delta)$. As a section of the fiber bundle, it satisfies $(\sigma,\check{\mu}(\sigma))\in E$ for every $\sigma\in\Delta$.
\begin{equation}
  \label{equ_gap_section}
  \check{\mu}(\sigma)=\sigma\circ\left(F(\sigma)-\left(\min_i F_i(\sigma)\right)\mathbf{1}\right)
\end{equation}

By Theorem~\ref{thm_section}, the section $\check{\mu}(\sigma)$ characterizes the gap of ${\rm VI}(F,\Delta)$; this theorem establishes the fundamental link between the fixed-point bundle $E$ and the variational inequality ${\rm VI}(F,\Delta)$.
\begin{theorem}
  \label{thm_section}
  For $(\sigma,\mu)\in\Delta\times\mathbb{R}^n$, the following equation holds.
  \begin{equation}
    \label{equ_measure}
    {\rm gap}(\sigma)=\mathbf{1}^\top\check{\mu}(\sigma)=\mathbf{1}^\top\mu-\min_i\left((G(\sigma,\mu)+\mu)/\sigma\right)_i
  \end{equation}
  Consequently, for $\sigma^\star\in\Delta$, the following statements are equivalent.
  \begin{enumerate}
    \item $\sigma^\star$ is a solution of ${\rm VI}(F,\Delta)$.
    \item $\sigma^\star$ is a zero point of $\check{\mu}(\sigma)$.
    \item $(\sigma^\star,\mathbf{0})$ is a limit point on $E\cap(\Delta\times\{\mu|\mu>0\})$.
  \end{enumerate}
\end{theorem}
\begin{proof}
  The gap function of ${\rm VI}(F,\Delta)$ satisfies
  $${\rm gap}(\sigma)=\sup_{\sigma'\in\Delta}\left\langle F(\sigma),\sigma-\sigma'\right\rangle=\sigma^\top F(\sigma)-\min_i F_i(\sigma).$$
  The section $\check{\mu}(\sigma)$ satisfies
  \begin{align*}
      & \mathbf{1}^\top\check{\mu}(\sigma)=\sigma^\top F(\sigma)-\min_i F_i(\sigma)                                                                                                                            \\
    = & \sigma^\top F(\sigma)-\min_i\left(\frac{G(\sigma,\mu)+\mu+(\sigma^\top F(\sigma)-\mathbf{1}^\top\mu)\sigma}{\sigma}\right)_i=\mathbf{1}^\top\mu-\min_i\left(\frac{G(\sigma,\mu)+\mu}{\sigma}\right)_i.
  \end{align*}
  This proves the equation.

  ${\rm (i)}\Leftrightarrow{\rm (ii)}$:
  Since ${\rm gap}(\sigma)\geq0$ and $\check{\mu}(\sigma)\geq0$, ${\rm gap}(\sigma)=0$ is equivalent to $\check{\mu}(\sigma)=0$.

  ${\rm (ii)}\Rightarrow{\rm (iii)}$:
  For any $\sigma\in\Delta$ and $k\in\mathbb{R}$, we have $(\sigma,\check{\mu}(\sigma)+k\sigma)\in E$. Continuity of $\check{\mu}(\sigma)$ gives $\check{\mu}(\sigma)\to\check{\mu}(\sigma^\star)$ as $\sigma\to\sigma^\star$. For any $\sigma_t\to\sigma^\star$ with $\sigma_t>0$ and any $k_t\to0$ with $k_t>0$, we have $\mu_t=\check{\mu}(\sigma_t)+k_t\sigma_t>0$ and $\mu_t\to0$. Hence $(\sigma^\star,\mathbf{0})$ is a limit point of $(\sigma_t,\mu_t)\in E$ on $\Delta\times\{\mu|\mu>0\}$.

  ${\rm (iii)}\Rightarrow{\rm (ii)}$:
  Given $(\sigma,\mu)\in E$, we have $G(\sigma,\mu)=0$. If additionally $\mu>0$, then $\min_i((G(\sigma,\mu)+\mu)/\sigma)_i>0$, and consequently $0\leq\mathbf{1}^\top\check{\mu}(\sigma)<\mathbf{1}^\top\mu$. If $(\sigma_t,\mu_t)\to(\sigma^\star,\mathbf{0})$ with $(\sigma_t,\mu_t)\in E$ and $\mu_t>0$, then $\mathbf{1}^\top\check{\mu}(\sigma_t)\to0$, hence $\check{\mu}(\sigma_t)\to0$ since $\check{\mu}(\sigma)\geq0$. Continuity of $\check{\mu}(\sigma)$ yields $\check{\mu}(\sigma^\star)=0$.
\end{proof}

Theorem~\ref{thm_section} reveals that solutions of ${\rm VI}(F,\Delta)$ are equivalently zero points of $\check{\mu}(\sigma)$ and equivalently limit points of the paths satisfying $G(\sigma,\mu)=0$ as $\mu\to\mathbf{0}^+$. This finally justifies path-following subject to $G(\sigma,\mu)=0$ (equivalently FPKKT~\eqref{equ_fpkkt}) as $\mu\to\mathbf{0}^+$ to solve ${\rm VI}(F,\Delta)$.

Equation~\eqref{equ_measure} also provides a measure of the gap of ${\rm VI}(F,\Delta)$. Specifically, for any $(\sigma,\mu)\in\Delta\times\{\mu|\mu>0\}$, if $\lvert G(\sigma,\mu)\rvert\leq\mu$ componentwise, then $0\leq\mathbf{1}^\top\check{\mu}(\sigma)\leq\mathbf{1}^\top\mu$. Thus $G(\sigma,\mu)$ measures the proximity to the fixed-point bundle and $\mathbf{1}^\top\mu$ measures the proximity to $\mu=0$. When both are sufficiently precise, so is $\mathbf{1}^\top\check{\mu}(\sigma)$, which translates precision in approaching the path and its endpoint into precision in solving ${\rm VI}(F,\Delta)$.

\subsection{Differential properties}

Having established that paths on the fixed-point bundle equivalently lead to the desired solutions, we now develop the differential properties needed to construct the path and the algorithm that follows it. To keep $\sigma\in\Delta$, we use the parameterization
\begin{equation}
  \sigma={\rm softmax}(\theta)=\frac{\exp(\theta)}{\mathbf{1}^\top\exp(\theta)}.
\end{equation}
The differential $d\sigma$ with respect to $d\theta$ is
\begin{equation}
  \label{equ_sigma_theta}
  \begin{aligned}
    d\sigma=\frac{\exp(\theta)\circ d\theta}{\mathbf{1}^\top\exp(\theta)}-\frac{\mathbf{1}^\top\left(\exp(\theta)\circ d\theta\right)}{(\mathbf{1}^\top\exp(\theta))^2} \exp(\theta)=\sigma\circ d\theta-(\sigma^\top d\theta)\sigma=\sigma\circ\left(I-\mathbf{1}\sigma^\top\right)d\theta.
  \end{aligned}
\end{equation}

For any $d\theta$ we have $\mathbf{1}^\top d\sigma=0$. Conversely, if $\mathbf{1}^\top d\sigma=0$, the linear equation $d\sigma/\sigma=(I-\mathbf{1}\sigma^\top)d\theta$ admits the general solution $d\theta=d\sigma/\sigma+k\mathbf{1}$ with $k\in\mathbb{R}$. We denote $\overline{d\theta}:=(I-\mathbf{1}\sigma^\top)d\theta$. Since adding a multiple of $\mathbf{1}$ to $d\theta$ does not change $d\sigma$, we use $\overline{d\theta}$ as the differential throughout this paper.

We introduce the differentiable section $\tilde{\mu}(\sigma)$ in equation~\eqref{equ_smooth_section}, which satisfies $\mathbf{1}^\top\tilde{\mu}(\sigma)=0$ and governs the differential properties of the fixed-point bundle. As a section of the fiber bundle, it satisfies $(\sigma,\tilde{\mu}(\sigma))\in E$ for every $\sigma\in\Delta$.
\begin{equation}
  \label{equ_smooth_section}
  \tilde{\mu}(\sigma)=\sigma\circ\left(F(\sigma)-\left(\sigma^\top F(\sigma)\right)\mathbf{1}\right)=\left(I-\sigma\mathbf{1}^\top\right)\left(\sigma\circ F(\sigma)\right)
\end{equation}

We use the section $\tilde{\mu}(\sigma)$ and its Jacobian to study the differential properties of the fixed-point bundle. The differential $d\tilde{\mu}(\sigma)$ with respect to $\overline{d\theta}$ is derived as follows; the third line uses $\mathbf{1}^\top d\sigma=0$, and the last line uses $d\sigma/\sigma=\overline{d\theta}=(I-\mathbf{1}\sigma^\top)\overline{d\theta}$.
\begin{align*}
  d\tilde{\mu}(\sigma) & =\left(I-\sigma\mathbf{1}^\top\right)\left(\sigma\circ\frac{\partial F(\sigma)}{\partial\sigma}+F(\sigma)\circ I\right)d\sigma-\mathbf{1}^\top\left(\sigma\circ F(\sigma)\right)d\sigma                                                   \\
                       & =\left(I-\sigma\mathbf{1}^\top\right)\left(\sigma\circ\frac{\partial F(\sigma)}{\partial\sigma}+F(\sigma)\circ I-(\sigma^\top F(\sigma)) I\right)d\sigma-\left(\sigma^\top F(\sigma)\right)\sigma\mathbf{1}^\top d\sigma                  \\
                       & =\sigma\circ\left(I-\mathbf{1}\sigma^\top\right)\left(\frac{\partial F(\sigma)}{\partial\sigma}\circ\sigma+\left(\left(I-\mathbf{1}\sigma^\top\right)F(\sigma)\right)\circ I\right)\frac{d\sigma}{\sigma}                                 \\
                       & =\sigma\circ\left(I-\mathbf{1}\sigma^\top\right)\left(\frac{\partial F(\sigma)}{\partial\sigma}\circ\sigma+\left(\left(I-\mathbf{1}\sigma^\top\right)F(\sigma)\right)\circ I\right)\left(I-\mathbf{1}\sigma^\top\right)\overline{d\theta}
\end{align*}

We introduce the matrix $J(\sigma)$ in equation~\eqref{equ_j}, where $k\in\mathbb{R}$ is an arbitrary scalar. This Jacobian-like matrix governs the differential properties directly and appears in the singular manifold, the differential equation, and the Newton equation of the fixed-point bundle.
\begin{equation}
  \label{equ_j}
  J(\sigma):=\left(I-\mathbf{1}\sigma^\top\right)\left(\frac{\partial F(\sigma)}{\partial\sigma}\circ\sigma+\left(\left(I-\mathbf{1}\sigma^\top\right)F(\sigma)\right)\circ I\right)\left(I-\mathbf{1}\sigma^\top\right)+k\mathbf{1}\sigma^\top
\end{equation}
Using $\sigma^\top\overline{d\theta}=0$, we obtain the differential equation of $\tilde{\mu}(\sigma)$:
\begin{equation}
  d\tilde{\mu}(\sigma)=\sigma\circ J(\sigma)\overline{d\theta}.
\end{equation}

The purpose of the arbitrary scalar $k\in\mathbb{R}$ in $J(\sigma)$ is to adjust the eigenvalues. When $k=0$, $J(\sigma)$ has an intrinsic zero eigenvalue with eigenvector $\mathbf{1}$. Proposition~\ref{thm_eigenvalue} shows how the eigenvalues are adjusted. Recall that a subspace $U$ is an invariant subspace of a matrix $A$ if $Au\in U$ for every $u\in U$, and that if $A$ admits an invariant subspace decomposition, its eigenvalues are the union of the eigenvalues of its restrictions to these invariant subspaces.
\begin{proposition}
  \label{thm_eigenvalue}
  Matrix $J(\sigma)$ admits the invariant subspace decomposition $\mathbb{R}^n={\rm span}(\mathbf{1})\oplus\ker(\sigma^\top)$. The eigenvalue corresponding to eigenvectors in ${\rm span}(\mathbf{1})$ is $k$, while eigenvalues corresponding to eigenvectors in $\ker(\sigma^\top)$ are independent of $k$.
\end{proposition}
\begin{proof}
  First, $J(\sigma)\mathbf{1}=k\mathbf{1}$ since $(I-\mathbf{1}\sigma^\top)\mathbf{1}=0$ and $\sigma^\top\mathbf{1}=1$; hence ${\rm span}(\mathbf{1})$ is an invariant subspace. Second, $\sigma^\top J(\sigma)=k\sigma^\top$ since $\sigma^\top(I-\mathbf{1}\sigma^\top)=0$ and $\sigma^\top\mathbf{1}=1$. Then $\ker(\sigma^\top)=\{u|\sigma^\top u=0\}$ is invariant because $\sigma^\top u=0$ implies $\sigma^\top J(\sigma)u=0$. Since ${\rm span}(\mathbf{1})$ and $\ker(\sigma^\top)$ are complementary, this is an invariant subspace decomposition.

  For an eigenvector in ${\rm span}(\mathbf{1})$, the corresponding eigenvalue is $k$ because $J(\sigma)\mathbf{1}=k\mathbf{1}$. For $u\in\ker(\sigma^\top)$, the equation $J(\sigma)u=\lambda u$ reduces to $J_0(\sigma)u=\lambda u$, where $J_0(\sigma)$ is the $k=0$ case of $J(\sigma)$, since $\sigma^\top u=0$. Thus the eigenvalue $\lambda$ is independent of $k$.
\end{proof}

Proposition~\ref{thm_eigenvalue} shows that the intrinsic zero eigenvalue of $J(\sigma)$ corresponding to the eigenvector $\mathbf{1}$ is removed, while the remaining eigenvalues are unchanged. This device prevents $J(\sigma)$ from being intrinsically singular, for theoretical simplicity; however, we use $J(\sigma)$ with $k=0$ in the algorithm for computational simplicity.

Next we derive the differential equation of $G(\sigma,\mu)=0$. This equation provides the tangent step in $\theta$ for a given reduction $d\mu$ along the fixed-point bundle.
\begin{align*}
  G'_\theta(\sigma,\mu)d\theta=                                                              & -G'_\mu(\sigma,\mu)d\mu                  \\
  \sigma\circ J(\sigma)\overline{d\theta}+(\mathbf{1}^\top\mu)\sigma\circ\overline{d\theta}= & \left(I-\sigma\mathbf{1}^\top\right)d\mu
\end{align*}
Multiplying both sides by $1/\sigma$, we obtain the differential equation of $G(\sigma,\mu)=0$ in equation~\eqref{equ_diff}. We write $J_G:=J(\sigma)+(\mathbf{1}^\top\mu)I$ as shorthand.
\begin{equation}
  \label{equ_diff}
  \left(J(\sigma)+(\mathbf{1}^\top\mu)I\right)\overline{d\theta}=\left(I-\mathbf{1}\sigma^\top\right)(d\mu/\sigma)
\end{equation}

A point $(\sigma,\mu)$ on the fixed-point bundle is called singular if $J(\sigma)+(\mathbf{1}^\top\mu)I$ is singular. For any $\sigma\in\Delta$, the fiber $B(\sigma)$ contains at most $n-1$ singular points, corresponding to the eigenvalues of $-J(\sigma)$ other than the arbitrary $-k$. The singular manifold of the fixed-point bundle is therefore constructed from the real eigenvalues of $J(\sigma)$. Let $\lambda_j(J(\sigma))\in\mathbb{R}$ denote the real eigenvalues of $J(\sigma)$ other than $k$. The set of all singular points of $E$ is formalized as the singular manifold $S$ in equation~\eqref{equ_singular}.
\begin{equation}
  \label{equ_singular}
  S=\left\{(\sigma,\tilde{\mu}(\sigma)-\lambda_j\left(J(\sigma)\right)\sigma)|\lambda_j\left(J(\sigma)\right)\in\mathbb{R},\sigma\in\Delta\right\}.
\end{equation}
Indeed, $S$ consists of all singular points of the fixed-point bundle. First, for every $\sigma\in\Delta$ and $(\sigma,\mu)\in S$, the point $\mu$ lies on the fiber $B(\sigma)$ over $\sigma$, because $\tilde{\mu}(\sigma)$ maps $\sigma$ to a point on $B(\sigma)$ and, if $\mu_0\in B(\sigma)$, then $\mu_0+k\sigma\in B(\sigma)$ for any $k\in\mathbb{R}$. Second, for every $\sigma\in\Delta$, we have $(\sigma,\mu)\in S$ if and only if $\mathbf{1}^\top\mu$ is an eigenvalue of $-J(\sigma)$: indeed, $\mathbf{1}^\top\tilde{\mu}(\sigma)=0$ and $\mathbf{1}^\top\sigma=1$, so $\mathbf{1}^\top\mu=-\lambda_j(J(\sigma))$, which are precisely the real eigenvalues of $-J(\sigma)$ other than $-k$.

\section{Paths on the fixed-point bundle}

Having shown that paths on the fixed-point bundle equivalently lead to solutions of ${\rm VI}(F,\Delta)$, and having prepared the differential properties of the bundle, we now establish a globally convergent path. We show that from any designated starting point there exists a path leading to solutions of ${\rm VI}(F,\Delta)$, and moreover that this path is well-behaved enough to be followed.

\subsection{Analytic curves}

Consider the curve on the fixed-point bundle $E$ defined by
\begin{equation}
  \Gamma(\sigma_{\rm init})=\left\{(\sigma,\mu)\in E|\mu=\gamma\sigma_{\rm init},\gamma\in[0,+\infty)\right\}
\end{equation}
with designated $\sigma_{\rm init}\in\Delta$. This is a one-dimensional curve parameterized by $\gamma$, and every $(\sigma,\mu)\in E$ on it satisfies $G(\sigma,\mu)=0$. When $F(\sigma)$ is real-analytic, so is $G(\sigma,\mu)$, and $\Gamma(\sigma_{\rm init})$ is called an analytic curve because it is the zero set of an analytic function.

Since $\sigma\circ F(\sigma)+v\sigma=\gamma\sigma_{\rm init}$ and $F(\sigma)$ is bounded, $\Gamma(\sigma_{\rm init})$ starts from a single point $(\sigma_{\rm init},\gamma\sigma_{\rm init})$ as $\gamma\to+\infty$. Thus $\Gamma(\sigma_{\rm init})$ is a path that starts from any designated $\sigma_{\rm init}\in\Delta$ as $\gamma\to+\infty$ and ends at solutions of ${\rm VI}(F,\Delta)$ as $\gamma\to0^+$.

The curve $\Gamma(\sigma_{\rm init})$ is smooth at nonsingular points, where path-following behaves well, but singular points may occur, and traditional path-following methods struggle at such points. The fixed-point bundle structure provides a simple mechanism for singularity avoidance through motion along the fiber $B(\sigma)$. Specifically, at a point $(\sigma,\mu)$ where $\mathbf{1}^\top\mu$ equals an eigenvalue of $-J(\sigma)$, we jump to $(\sigma,\mu+\beta\sigma)$ with $\beta$ chosen so that $\mu+\beta\sigma>0$ and $\mathbf{1}^\top\mu+\beta$ is not an eigenvalue of $-J(\sigma)$. Then $(\sigma,\mu+\beta\sigma)$ is no longer singular, and we can proceed to reduce $\mu+\beta\sigma$. This moves us to a new analytic curve $\Gamma((\mu+\beta\sigma)/(\mathbf{1}^\top\mu+\beta))$ with a different starting point.

Analogously to the classical oddness theorem for Nash equilibria \cite{oddness2}, there is an oddness theorem for ${\rm VI}(F,\Delta)$.
\begin{theorem}[Oddness theorem]
  \label{thm_odd}
  For every $\mu>0$, if $J(\sigma)+(\mathbf{1}^\top\mu)I$ is nonsingular at every solution of $G(\sigma,\mu)=0$, then there are an odd number of solutions.
\end{theorem}
\begin{proof}
  The proof follows the fixed-point index argument developed in \cite{fp_degree}.

  Let $M_\mu$ denote the map from $\sigma$ to $\hat{\sigma}$ given by the Brouwer function $(\hat{\sigma},r,v)=M(\sigma,\mu)$. Since $M_\mu$ is a continuous function that maps the compact convex set $\Delta$ to itself, its global fixed-point index on $\Delta$ is $i(M_\mu,\Delta)=1$.

  For a fixed point $\sigma$ of $M_\mu$, if $I-DM_\mu(\sigma)$ is nonsingular, then the local fixed-point index of $\sigma$ is given by the Lefschetz-Hopf formula $${\rm ind}(M_\mu,\sigma)={\rm sign}(\det(I-DM_\mu(\sigma)))\in\{+1,-1\}.$$

  As Theorem~\ref{thm_limit} shows, the fixed-point equation $\hat{\sigma}=\sigma$ subject to $(\hat{\sigma},r,v)=M(\sigma,\mu)$ and the fixed-point bundle equation $G(\sigma,\mu)=0$ are equivalent through linear transformations. In addition, since $\mu>0$, we have $\sigma>0$ and $r>0$, so the linear transformations are invertible. Thus, the singularity of $I-DM_\mu(\sigma)$ is equivalent to that of $J(\sigma)+(\mathbf{1}^\top\mu)I$.

  Nonsingular solutions are isolated by the inverse function theorem. Since the solutions lie in the compact simplex $\Delta$, there are finitely many of them; otherwise there would be an accumulation point, contradicting isolation.

  Finally, by the additivity property of the fixed-point index, when the number of fixed points is finite, the global index equals the sum of the local indices. Thus, a finite number of $1$ or $-1$ sums to $1$, which implies that the number must be odd.
\end{proof}

By Sard's theorem \cite{sard1942measure}, for almost all $\mu>0$, all solutions of $G(\sigma,\mu)=0$ are nonsingular. Thus, Theorem~\ref{thm_odd} also implies that the analytic curve $\Gamma(\sigma_{\rm init})$ almost always consists of finitely many branches. By the Newton-Puiseux theorem \cite{puiseux}, every branch of $\Gamma(\sigma_{\rm init})$ has a unique analytic continuation through every point on it. Thus, the only boundary points of $\Gamma(\sigma_{\rm init})$ occur in the limits $\gamma\to0^+$ and $\gamma\to+\infty$: the former correspond to solutions of ${\rm VI}(F,\Delta)$, while the latter is the unique starting point $\sigma_{\rm init}$. Therefore, if the solutions of ${\rm VI}(F,\Delta)$ are all nonsingular, the branches of $\Gamma(\sigma_{\rm init})$ connect the solutions and the unique starting point $\sigma_{\rm init}$ in pairs.

\begin{figure}
  \centering
  \includegraphics[width=0.6\textwidth]{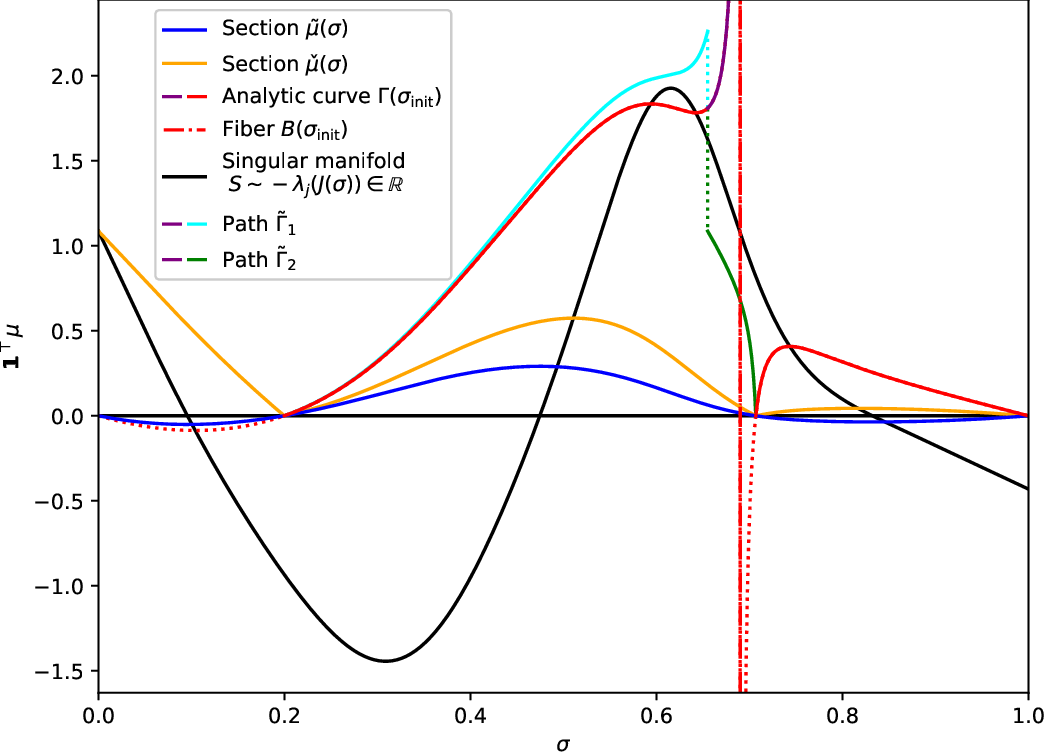}
  \caption{
    The fixed-point bundle $E$ for a two-dimensional ${\rm VI}(F,\Delta)$. The horizontal axis represents the simplex coordinate $\sigma$, and the vertical axis represents the fiber coordinate $\mathbf{1}^\top\mu$. The smooth section $\tilde{\mu}(\sigma)$ (the first component plotted since $\mathbf{1}^\top\tilde{\mu}(\sigma)=0$) induces the differential structure through $J(\sigma)$, while the nonnegative section $\check{\mu}(\sigma)$ (the nonzero component plotted since $\min_i\check{\mu}_i(\sigma)=0$) equals the gap function of ${\rm VI}(F,\Delta)$. Analytic curves $\Gamma(\sigma_{\rm init})$ connect a designated starting point to solutions as $\mu\to\mathbf{0}^+$. Singular points, where path-following cannot proceed, lie on the singular manifold $S$, where $\mathbf{1}^\top\mu$ equals a real eigenvalue of $-J(\sigma)$. The path-following trajectory $\tilde{\Gamma}$ bypasses these singularities by moving along fibers, illustrated by a positive jump ($\tilde{\Gamma}_1$) and a negative jump ($\tilde{\Gamma}_2$). After finitely many such fiber jumps, $\tilde{\Gamma}$ reaches a solution of ${\rm VI}(F,\Delta)$.
  }
  \label{fig1}
\end{figure}

\begin{figure}
  \centering
  \includegraphics[width=0.6\textwidth]{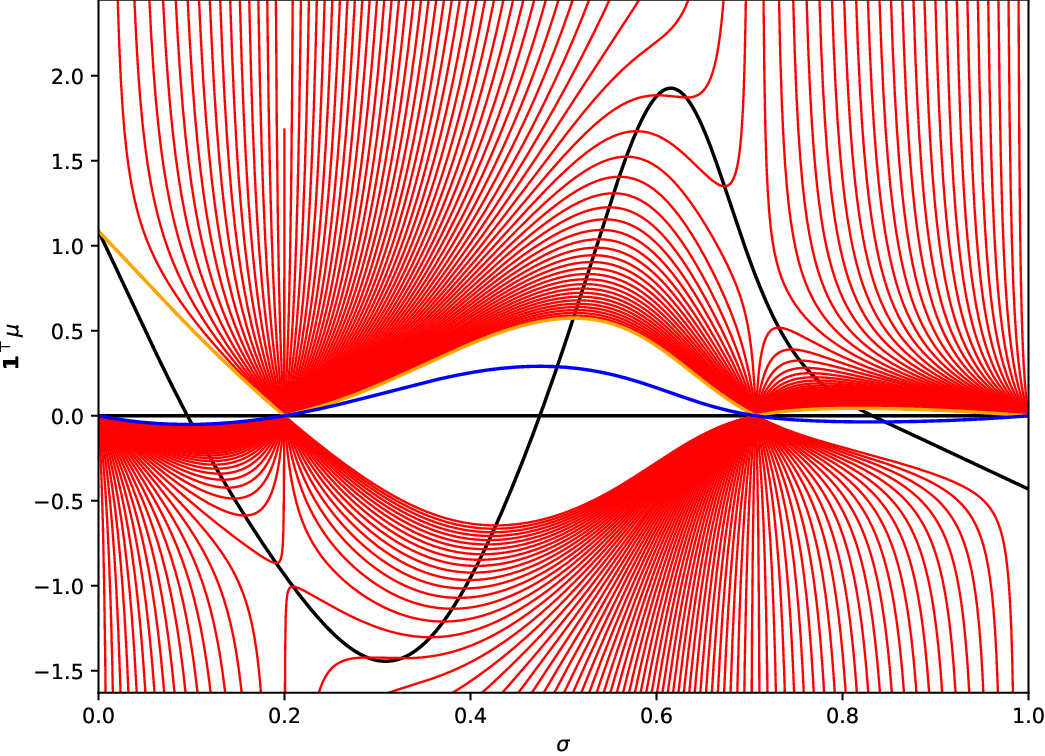}
  \caption{
    Analytic curves on the fixed-point bundle $E$ with different starting points $\sigma_{\rm init}$. For each $\mu>0$, there is exactly one analytic curve passing through it. For each analytic curve, its intersections with the singular manifold, which is given by the real eigenvalues of $-J(\sigma)$, are singular points. Each fiber jump switches to another analytic curve with a different starting point.
  }
  \label{fig2}
\end{figure}

\subsection{The path}

The path we follow is derived from the analytic curve $\Gamma(\sigma_{\rm init})$ and singularity avoidance along the fibers. This is formalized as $\tilde{\Gamma}$ in equation~\eqref{equ_path}. The key idea is that $\tilde{\Gamma}$ is obtained by concatenating segments of analytic curves separated by fiber jumps. The starting point $(\sigma_0,\gamma_0\mu_0)\in E$ is given by $\gamma_0\mu_0=v_{\rm init}\sigma_{\rm init}$, with designated $\sigma_{\rm init}$ and sufficiently large $v_{\rm init}$. Then $\tilde{\Gamma}$ consists of segments of different analytic curves $\Gamma(\sigma_{\rm init})$, concatenated in $\sigma$ but jumping along fibers in $\mu$. It eventually ends at $(\sigma^\star,\mathbf{0})$ with $\mathbf{1}^\top\mu=0$, where $\sigma^\star$ is a solution of ${\rm VI}(F,\Delta)$.
\begin{equation}
  \label{equ_path}
  \begin{aligned}
    \tilde{\Gamma}=\bigcup_{q=1}^{Q} \left\{(\sigma,\mu)\in E|\mu=(1-\gamma)\hat{\mu}_{q-1}+\gamma\mu_q,\gamma\in[0,1]\right\}, & \\
    \hat{\mu}_{q-1}=\mu_{q-1}+\beta_q\sigma_{q-1},~\mu_0=v_{\rm init}\sigma_{\rm init},~\mu_Q=\mathbf{0}                        & \\
  \end{aligned}
\end{equation}

We would like $\tilde{\Gamma}\cap S=\varnothing$ for singularity avoidance, where $S$ is the singular manifold. However, $\mathbf{1}^\top\mu=0$ is unavoidable at the solution $(\sigma^\star,\mathbf{0})$. If $J(\sigma^\star)$ is singular, then $(\sigma^\star,\mathbf{0})$ is an unavoidable singular point, that is, $\tilde{\Gamma}$ must intersect $S$. Thus we instead require $\tilde{\Gamma}\cap S\subseteq \{(\sigma^\star,\mathbf{0})\}$ and only avoid the nonzero eigenvalues of $-J(\sigma)$. In fact, this can be achieved with only finitely many segments in $\tilde{\Gamma}$.

\begin{proposition}
  \label{thm_path}
  There exists $\tilde{\Gamma}$ consisting of finitely many segments satisfying $\tilde{\Gamma}\cap S\subseteq \{(\sigma^\star,\mathbf{0})\}$.
\end{proposition}
\begin{proof}
  Every real-analytic curve on a compact domain has only finitely many singular points on it, as long as it is not constantly singular. Thus, there are finitely many singular points on the analytic curve $\Gamma(\sigma_{\rm init})$, because it is not constantly singular, since it is nonsingular where $\gamma$ is large. For each singular point $(\sigma_s,\gamma_s\sigma_{\rm init})\in\Gamma(\sigma_{\rm init})$, let $d_1,d_2$ be two small scalars, and let $(\sigma_1,(1+d_1)\gamma_s\sigma_{\rm init})$ and $(\sigma_2,(1+d_2)\gamma_s\sigma_{\rm init})$ be two points on $\Gamma(\sigma_{\rm init})$ before and after the singular point. We adopt the following path to circumvent the singular point $(\sigma_s,\gamma_s\sigma_{\rm init})$ on $\Gamma(\sigma_{\rm init})$.
  \begin{enumerate}
    \item The fiber jump between $(\sigma_1,(1+d_1)\gamma_s\sigma_{\rm init})\in\Gamma(\sigma_{\rm init})$ and $(\sigma_1,(1+d_1)(\gamma_s\sigma_{\rm init}+\beta\sigma_1))$.
    \item The segment changing $\mu$ from $(1+d_1)(\gamma_s\sigma_{\rm init}+\beta\sigma_1)$ to $(1+d_2)(\gamma_s\sigma_{\rm init}+\beta\sigma_2)$.
    \item The fiber jump between $(\sigma_2,(1+d_2)\gamma_s\sigma_{\rm init})\in\Gamma(\sigma_{\rm init})$ and $(\sigma_2,(1+d_2)(\gamma_s\sigma_{\rm init}+\beta\sigma_2))$.
  \end{enumerate}

  For the singular point $(\sigma_s,\gamma_s\sigma_{\rm init})$, there exists $\beta$ such that $(\sigma_s,\gamma_s\sigma_{\rm init}+\beta\sigma_s)$ is a nonsingular point, because there are finitely many singular points on each fiber. Since $(\sigma_s,\gamma_s\sigma_{\rm init}+\beta\sigma_s)$ is a nonsingular point, it has a nonsingular neighborhood by continuity.

  As $d_1\to0$, we have $\sigma_1\to\sigma_s$, and then $(\sigma_1,(1+d_1)(\gamma_s\sigma_{\rm init}+\beta\sigma_1))\to(\sigma_s,\gamma_s\sigma_{\rm init}+\beta\sigma_s)$. As $d_2\to0$, we have $\sigma_2\to\sigma_s$, and then $(\sigma_2,(1+d_2)(\gamma_s\sigma_{\rm init}+\beta\sigma_2))\to(\sigma_s,\gamma_s\sigma_{\rm init}+\beta\sigma_s)$. Thus, there exist sufficiently small $|d_1|>0$ and $|d_2|>0$ such that the segment in step (ii) lies in the nonsingular neighborhood of $(\sigma_s,\gamma_s\sigma_{\rm init}+\beta\sigma_s)$.

  Therefore, it takes one segment and two fiber jumps to circumvent a singular point on $\Gamma(\sigma_{\rm init})$ while sticking to $\Gamma(\sigma_{\rm init})$ before and after the singular point. Since there are finitely many singular points on $\Gamma(\sigma_{\rm init})$, we obtain a path consisting of finitely many segments satisfying $\tilde{\Gamma}\cap S\subseteq \{(\sigma^\star,\mathbf{0})\}$.
\end{proof}

Therefore, $\tilde{\Gamma}\cap S\subseteq \{(\sigma^\star,\mathbf{0})\}$ with finitely many segments, where $\tilde{\Gamma}\cap S=\varnothing$ if $J(\sigma^\star)$ is nonsingular, and $\tilde{\Gamma}\cap S=\{(\sigma^\star,\mathbf{0})\}$ if $J(\sigma^\star)$ is singular. In other words, except possibly at the solution, singular points can always be bypassed through finitely many segments, so singularities do not obstruct global convergence.

In particular, the path in Proposition~\ref{thm_path} follows the analytic curve $\Gamma(\sigma_{\rm init})$ except for circumventing finitely many singular points on it. Note that $\mathbf{1}^\top\mu$ is not monotone along $\Gamma(\sigma_{\rm init})$, since the direction in which it varies may change after crossing a singular point. At each $(\sigma,\mu)\in\Gamma(\sigma_{\rm init})$, the direction of variation of $\mathbf{1}^\top\mu$ along $\Gamma(\sigma_{\rm init})$ is determined by ${\rm sign}(\det(J(\sigma)+(\mathbf{1}^\top\mu)I))$. That is, to follow $\Gamma(\sigma_{\rm init})$, $\mathbf{1}^\top\mu$ must decrease when ${\rm sign}(\det(J_G))=1$ and must increase when ${\rm sign}(\det(J_G))=-1$.

Finally, Proposition~\ref{thm_path} is an existence result. In fact, we can fully exploit the geometric freedom of the fixed-point bundle without necessarily following the path constructed in its proof. As long as the path leads to a solution with finitely many segments, the subsequent analysis remains valid.

\subsection{A compact nonsingular neighborhood for path-following}

In addition to $\tilde{\Gamma}$ avoiding $S$, we also need a neighborhood of $\tilde{\Gamma}$ that avoids $S$, because the corrector is a subiteration that converges onto the path $\tilde{\Gamma}$ within its neighborhood. We introduce the neighborhood $\mathcal{N}_{\tilde{\Gamma}}(\zeta)$ in equation~\eqref{equ_neighborhood}. Its purpose is to provide a compact nonsingular neighborhood on which path-following is well-behaved, allowing us to bound the relevant quantities.
\begin{equation}
  \label{equ_neighborhood}
  \mathcal{N}_{\tilde{\Gamma}}(\zeta)={\rm Closure}\left(\left\{({\rm softmax}(\theta),\mu)|\left\lVert\theta-\ln\sigma_\mu\right\rVert\leq d,(\sigma_\mu,\mu)\in\tilde{\Gamma},\mathbf{1}^\top\mu>\zeta\right\}\right)
\end{equation}

The set $\mathcal{N}_{\tilde{\Gamma}}(\zeta)$ is the union, over all $(\sigma_\mu,\mu)\in\tilde{\Gamma}$ with $\mathbf{1}^\top\mu>\zeta$, of $d$-radius neighborhoods of $\sigma_\mu$ in the $\ln\sigma$-coordinate. Here $\zeta\geq0$ is the prescribed precision, which determines whether path-following ends exactly at $(\sigma^\star,\mathbf{0})$ or slightly before it if the solution is singular. Taking the closure ensures compactness.

The $\ln\sigma$-coordinate in the definition of $\mathcal{N}_{\tilde{\Gamma}}(\zeta)$ aligns with the parameterization $\sigma={\rm softmax}(\theta)$, yielding a neighborhood of $\sigma_\mu$ that remains well-behaved near the simplex boundary. Note that $\lVert\theta-\ln\sigma_\mu\rVert\leq d$ is equivalent to $\mathrm{e}^{-d}-1\leq(\sigma_\mu-\sigma)/\sigma\leq \mathrm{e}^d-1$. The $\ln\sigma$-coordinate strongly resembles the local norm $\lVert d\sigma\rVert_{{\rm diag}(1/\sigma^2)}=d\sigma^\top{\rm diag}(1/\sigma^2)d\sigma$ used in interior-point methods, where ${\rm diag}(1/\sigma^2)$ is the Hessian of the log barrier $-\ln\sigma$. In interior-point methods, the neighborhood defined by the Hessian-induced local norm is called the Dikin ellipsoid \cite{dikin}.

We prove $\mathcal{N}_{\tilde{\Gamma}}(\zeta)\cap S=\varnothing$ using compactness arguments. In the subsequent developments, we repeatedly use that the image of a compact set under a continuous function is compact, and that compact sets are bounded.
\begin{proposition}
  \label{thm_neighborhood}
  Let $\zeta\neq0$ if $J(\sigma^\star)$ is singular. Then for any $\zeta\geq0$, there exists $d>0$ such that $\mathcal{N}_{\tilde{\Gamma}}(\zeta)\cap S=\varnothing$.
\end{proposition}
\begin{proof}
  Denote $\tilde{\Gamma}'=\{(\sigma,\mu)\in\tilde{\Gamma}|\mathbf{1}^\top\mu\geq\zeta\}$. If $J(\sigma^\star)$ is nonsingular, $\zeta=0$ and $\tilde{\Gamma}'=\tilde{\Gamma}$. If $J(\sigma^\star)$ is singular, $\zeta>0$ and $\tilde{\Gamma}'$ is the truncation of $\tilde{\Gamma}$ at $\mathbf{1}^\top\mu=\zeta$. Thus $\tilde{\Gamma}'\cap S=\varnothing$; then, if $\tilde{\Gamma}'$ and $S$ are compact, there is a positive gap between them.

  Each analytic curve segment in $\tilde{\Gamma}'$ is compact as the image of a compact interval of $\gamma$ under a continuous map. A finite union of compact sets is compact, so $\tilde{\Gamma}'$ is compact. The map $\sigma\mapsto(\sigma,\tilde{\mu}(\sigma)-\lambda_j(J(\sigma))\sigma)$ is continuous because $\tilde{\mu}(\sigma)$ is continuous and the eigenvalue mapping $\lambda_j(J(\sigma))$ is continuous. Hence $S$ is compact as the image of the compact set $\Delta$ under a continuous function.

  Since $\tilde{\Gamma}'\cap S=\varnothing$ and both sets are compact, there exists $d_0>0$ such that for any $(\sigma_\mu,\mu)\in\tilde{\Gamma}'$ and $(\sigma_s,\mu_s)\in S$,
  $$\left\lVert(\sigma_\mu,\mu)-(\sigma_s,\mu_s)\right\rVert\geq d_0.$$

  Let $\mathcal{N}_{\tilde{\Gamma}}^\prime(\zeta)$ be the set that differs from $\mathcal{N}_{\tilde{\Gamma}}(\zeta)$ only in that the closure operation is not taken. For any $(\sigma_n,\mu)\in\mathcal{N}_{\tilde{\Gamma}}^\prime(\zeta)$, there exists $(\sigma_\mu,\mu)\in\tilde{\Gamma}'$ such that
  \begin{align*}
         & \left\lVert(\sigma_n,\mu)-(\sigma_\mu,\mu)\right\rVert=\left\lVert{\rm softmax}(\theta_n)-{\rm softmax}(\ln\sigma_\mu)\right\rVert                   \\
    \leq & \left(\sup_{\sigma\in\Delta}\left\lVert\sigma\circ(I-\mathbf{1}\sigma^\top)\right\rVert\right)\left\lVert\theta_n-\ln\sigma_\mu\right\rVert\leq d/2,
  \end{align*}
  where $\sigma\circ(I-\mathbf{1}\sigma^\top)$ is the Jacobian of ${\rm softmax}$, whose norm is bounded by $1/2$ on $\Delta$.

  For any $(\sigma_n,\mu)\in\mathcal{N}_{\tilde{\Gamma}}^\prime(\zeta)$ and $(\sigma_s,\mu_s)\in S$,
  $$\left\lVert(\sigma_n,\mu)-(\sigma_s,\mu_s)\right\rVert\geq\left\lVert(\sigma_\mu,\mu)-(\sigma_s,\mu_s)\right\rVert-\left\lVert(\sigma_n,\mu)-(\sigma_\mu,\mu)\right\rVert\geq d_0-d/2.$$

  Choose $d\in(0,2d_0)$. Then $\mathcal{N}_{\tilde{\Gamma}}^\prime(\zeta)$ and $S$ have positive gap $d_0-d/2>0$. If two sets have positive gap, their closures also have positive gap. Hence $\mathcal{N}_{\tilde{\Gamma}}(\zeta)\cap S=\varnothing$.
\end{proof}

The main purpose of constructing $\mathcal{N}_{\tilde{\Gamma}}(\zeta)$ is to bound two quantities that will appear in the algorithm. Let $p_\Delta(\mathcal{N}_{\tilde{\Gamma}}(\zeta))$ denote the projection of $\mathcal{N}_{\tilde{\Gamma}}(\zeta)$ onto the $\sigma$-space.
\begin{lemma}
  \label{thm_bounds}
  Let $\zeta\neq0$ if $J(\sigma^\star)$ is singular. Then for any $\zeta\geq0$, the following quantities are bounded on $\mathcal{N}_{\tilde{\Gamma}}(\zeta)$.
  \begin{enumerate}
    \item The smallest singular value $s_{\min}(J(\sigma)+(\mathbf{1}^\top\mu) I)$ attains a lower bound on $\mathcal{N}_{\tilde{\Gamma}}(\zeta)$.
    \item $\lVert\mu/\sigma\rVert$ attains an upper bound on $\mathcal{N}_{\tilde{\Gamma}}(\zeta)$.
  \end{enumerate}
\end{lemma}
\begin{proof}
  (i) Given that $\partial F/\partial\sigma$ is continuous on $p_\Delta(\mathcal{N}_{\tilde{\Gamma}}(\zeta))$, the matrix $J(\sigma)+(\mathbf{1}^\top\mu) I$ is continuous on $\mathcal{N}_{\tilde{\Gamma}}(\zeta)$. Then $s_{\min}(J(\sigma)+(\mathbf{1}^\top\mu) I)$ is continuous on the compact set $\mathcal{N}_{\tilde{\Gamma}}(\zeta)$, hence attains a minimum.

  (ii) We have
  $$\mu/\sigma=(\sigma_\mu/\sigma)\circ r_\mu=(\sigma_\mu/\sigma)\circ(F(\sigma_\mu)+v_\mu\mathbf{1}).$$
  The ratio is bounded by $\mathrm{e}^{-d}\leq\sigma_\mu/\sigma\leq \mathrm{e}^d$ on $\mathcal{N}_{\tilde{\Gamma}}(\zeta)$. The term $F(\sigma_\mu)$ is bounded by continuity on the compact set $p_\Delta(\mathcal{N}_{\tilde{\Gamma}}(\zeta))$. The term $v_\mu$ is bounded by continuity of the Brouwer function $(\hat{\sigma},r,v)=M(\sigma,\mu)$ on the compact set $\mathcal{N}_{\tilde{\Gamma}}(\zeta)$. Indeed, $(\hat{\sigma},r,v)=M(\sigma,\mu)$ is continuous on the part of $\mathcal{N}_{\tilde{\Gamma}}(\zeta)$ where $\mu>0$, by the same argument as in Proposition~\ref{thm_brouwer_continuous}, and the limit as $\mu\to\mathbf{0}^+$ exists and yields $v\to-\min_i F_i(\sigma)$.
\end{proof}

In Lemma~\ref{thm_bounds}, $p_\Delta(\mathcal{N}_{\tilde{\Gamma}}(\zeta))$ denotes the projection of $\mathcal{N}_{\tilde{\Gamma}}(\zeta)$ onto the $\sigma$-space. Since we only need bounds on $\mathcal{N}_{\tilde{\Gamma}}(\zeta)$, continuous differentiability of $F$ on $p_\Delta(\mathcal{N}_{\tilde{\Gamma}}(\zeta))$ suffices.

Denote the lower bound of $s_{\min}(J(\sigma)+(\mathbf{1}^\top\mu) I)$ on $\mathcal{N}_{\tilde{\Gamma}}(\zeta)$ by $s_{\min}(J_G)$. If $J(\sigma^\star)$ is nonsingular, this bound is independent of $\zeta$ and uniform over the entire path-following process until reaching $(\sigma^\star,\mathbf{0})$. If $J(\sigma^\star)$ is singular, the bound is uniform only over the part where $\mathbf{1}^\top\mu\geq\zeta>0$, and $s_{\min}(J_G)\to0$ as $\zeta\to0$. Thus, for a singular solution, the subsequent convergence result applies only to path-following that ends at $\mathbf{1}^\top\mu=\zeta$ with a prescribed $\zeta>0$.

\section{Path-following on the fixed-point bundle}

\subsection{The predictor-corrector framework}

The fixed-point bundle provides a path connecting an arbitrary starting point to a solution of ${\rm VI}(F,\Delta)$. We now translate this geometric structure into a predictor-corrector algorithm, shown in Algorithm~\ref{algo}. The predictor follows the tangent direction of the bundle, the corrector restores feasibility with respect to the bundle equation, and singularities are avoided by moving along bundle fibers.

\begin{algorithm}
  \caption{Path-following on the fixed-point bundle}
  \label{algo}
  \begin{algorithmic}[1]
    \Require A smooth map $F:\Delta\to\mathbb{R}^n$ and its derivative $\partial F/\partial\sigma$, a designated starting point $\sigma_{\rm init}\in\Delta$, and a desired precision $\epsilon>0$
    \State Set $(\sigma_0,\mu_0)=(\sigma_{\rm init},v_{\rm init}\sigma_{\rm init})$ for a sufficiently large $v_{\rm init}$
    \Repeat
    \Repeat
    \State Compute $(\hat{\sigma}_k,r_k,v_k)=M(\sigma_k,\mu_t)$ by solving bisection problem \eqref{equ_bisection}
    \State Construct the matrix $J(\sigma_k)$ in equation~\eqref{equ_j}
    \State Solve regularized Newton equation~\eqref{equ_regu_kkt} or~\eqref{equ_regu_barr} for $\overline{d\theta}_k$ with $\delta_H=\lVert G(\sigma_k,\mu_t)\rVert/n$
    \State Update $\sigma_{k+1}={\rm softmax}(\ln\sigma_k+\overline{d\theta}_k)$
    \Until{$\lVert G(\sigma_k,\mu_t)\rVert\leq\epsilon/n$}
    \State Set $\hat{\mu}_t=\mu_t+\beta_t\sigma_k>0$ such that $\mathbf{1}^\top\mu_t+\beta_t$ avoids the nonzero eigenvalues of $-J(\sigma_k)$
    \State Update $\mu_{t+1}=(\hat{\mu}_t-\eta_t u_t).{\rm clip}(\min=\epsilon/n)$ with a small $\eta_t$ and a direction $u_t$
    \State Solve differential equation~\eqref{equ_diff} for $\overline{d\theta}_t$ with $d\mu_t=\mu_{t+1}-\hat{\mu}_t$
    \State Set $\sigma_0={\rm softmax}(\ln\sigma_k+\overline{d\theta}_t)$
    \Until{$\mathbf{1}^\top\check{\mu}(\sigma_k)\leq\epsilon$}\\
    \Return $\sigma_k$ as an approximate solution of ${\rm VI}(F,\Delta)$
  \end{algorithmic}
\end{algorithm}

\textbf{Predictor:}
The predictor first performs singularity avoidance. When a singular point is approached, the predictor moves along the fiber by adding $\beta\sigma$ to $\mu$, jumping to a nonsingular point. It then moves in the tangent direction given by differential equation~\eqref{equ_diff} with $d\mu=-\eta u$, where $u$ is the direction vector parallel to $\hat{\mu}_{q-1}-\mu_q$ from the finitely many segments of the path $\tilde{\Gamma}$. In practice, we simply use $u_t=\pm\hat{\mu}_t$, following whichever analytic curve the fiber jump switches to, without strictly following the proof of Proposition~\ref{thm_path} to return to the original analytic curve; this works well in our experiments. In the subsequent convergence analysis, however, we still use the free directions as defined by $\tilde{\Gamma}$.

\textbf{Corrector:}
The corrector solves the fixed-point bundle equation $G(\sigma,\mu)=0$ for a fixed $\mu$. In interior-point methods, the perturbed KKT conditions and the barrier problem are derived from the same original problem. However, here they do not entirely align, so we derive two different but closely related correctors from them. In the following subsections, we introduce two regularized Newton correctors for practical use and prove the convergence of the overall algorithm with the standard Newton corrector.

\textbf{Stopping criterion:}
The outer iteration terminates when $\mathbf{1}^\top\check{\mu}(\sigma)$ falls below a prescribed tolerance $\epsilon$, which guarantees precision on the gap of ${\rm VI}(F,\Delta)$. During the predictor steps, we enforce $\min_i\mu_i\geq\epsilon/n$, so that $\lVert G(\sigma,\mu)\rVert\leq\epsilon/n$ in the inner iteration suffices to imply $\lvert G(\sigma,\mu)\rvert\leq\mu$ in the corrector loop. When $\mathbf{1}^\top\mu$ falls below $\epsilon$, Theorem~\ref{thm_section} guarantees $\mathbf{1}^\top\check{\mu}(\sigma)\leq\epsilon$.

\subsection{Corrector derived from perturbed KKT conditions}

By perturbed KKT conditions, we mean the fixed-point bundle equation $G(\sigma,\mu)=0$ that is equivalent to FPKKT~\eqref{equ_fpkkt}. The Newton equation of $G(\sigma,\mu)=0$ is derived as follows, where $G'_\theta(\sigma,\mu)d\theta$ was obtained in differential equation~\eqref{equ_diff}.
\begin{align*}
  G'_\theta(\sigma,\mu)d\theta=                                              & -G(\sigma,\mu)                                                   \\
  \sigma\circ\left(J(\sigma)+(\mathbf{1}^\top\mu)I\right)\overline{d\theta}= & -\left(I-\sigma\mathbf{1}^\top\right)(\sigma\circ F(\sigma)-\mu)
\end{align*}
Multiplying both sides by $1/\sigma$, we obtain the Newton equation of $G(\sigma,\mu)=0$ in equation~\eqref{equ_newton}. In addition, we denote the term on the right-hand side as $\tilde{G}(\sigma,\mu):=G(\sigma,\mu)/\sigma$.
\begin{equation}
  \label{equ_newton}
  \left(J(\sigma)+(\mathbf{1}^\top\mu)I\right)\overline{d\theta}=-\left(I-\mathbf{1}\sigma^\top\right)(F(\sigma)-\mu/\sigma)
\end{equation}

The standard Newton iteration is simple for theoretical analysis, but its performance in practice is not ideal. The step length is affected by the smallest singular value $s_{\min}(J_G)$, which makes the iteration appear unstable if $s_{\min}(J_G)$ is too small. A classical remedy is the regularized Newton equation~\eqref{equ_regu_kkt}, derived from the standard Newton equation~\eqref{equ_newton} by multiplying both sides by $J_G^\top$ and adding the regularization term $\delta_H I$ with $\delta_H>0$. The coefficient then becomes positive semidefinite and is controlled by $\delta_H$. Indeed, the smallest singular value of the coefficient is now $s_{\min}^2(J_G)+\delta_H$, and the step length can be controlled with a mild $\delta_H$. However, convergence reduces from quadratic to linear with rate $\delta_H/(s_{\min}^2(J_G)+\delta_H)$, so regularization trades convergence speed for iteration stability.
\begin{equation}
  \label{equ_regu_kkt}
  \begin{aligned}
    \left(J_G^\top J_G+\delta_H I\right)d\theta= & -J_G^\top\tilde{G}(\sigma,\mu)              \\
    \overline{d\theta}=                          & \left(I-\mathbf{1}\sigma^\top\right)d\theta \\
  \end{aligned}
\end{equation}

Regularized Newton equation~\eqref{equ_regu_kkt} is approximately the regularized Gauss-Newton equation of the least squares problem~\eqref{equ_square_kkt} in terms of the Euclidean norm of $\tilde{G}(\sigma,\mu)$, which is equivalently the local norm of $G(\sigma,\mu)$. Since the local norm is induced by the Hessian matrix ${\rm diag}(1/\sigma^2)$ of the log barrier $-\ln\sigma$, problem~\eqref{equ_square_kkt} is actually a log barrier problem. Thus gradient descent with its inexact gradient $J_G^\top\tilde{G}(\sigma,\mu)$ also serves as a corrector for solving $G(\sigma,\mu)=0$. Regularized Newton equation~\eqref{equ_regu_kkt} mixes Newton and gradient steps: it reduces to the standard Newton step~\eqref{equ_newton} if $\delta_H=0$, and approximates the gradient step $-(1/\delta_H)J_G^\top\tilde{G}(\sigma,\mu)$ if $\delta_H$ is large.
\begin{equation}
  \label{equ_square_kkt}
  \min_\sigma\frac{1}{2}\left\lVert\tilde{G}(\sigma,\mu)\right\rVert=\min_\sigma\frac{1}{2}\left\lVert G(\sigma,\mu)\right\rVert _{{\rm diag}(1/\sigma^2)}
\end{equation}

\subsection{Corrector derived from barrier problem}

Since the fixed-point bundle is also a special central path of the MCP~\eqref{equ_mcp}, we can also derive a gradient from the barrier MCP~\eqref{equ_mcp} for the corrector. First, we additionally require that $(r,v)$ satisfy the Brouwer function $(\hat{\sigma},r,v)=M(\sigma,\mu)$. Then $\sigma$ becomes the only optimization variable of the barrier MCP~\eqref{equ_mcp}, with $(r,v)$ as an intermediate variable. The gradient is taken only with respect to $d\sigma$, and hence only with respect to $\overline{d\theta}$. In addition, this enables us to use $\mathbf{1}^\top(\sigma-\hat{\sigma})=0$ to eliminate $dv$ from the differential.

The gradient of the barrier MCP~\eqref{equ_mcp} is derived as follows. In the third line, we use $\mu/r=\hat{\sigma}$, $dv\mathbf{1}^\top(\sigma-\hat{\sigma})=0$, and $d\sigma/\sigma=\overline{d\theta}$, and we also use $(I-\sigma\mathbf{1}^\top)(\sigma-\hat{\sigma})=(\sigma-\hat{\sigma})$ and $\overline{d\theta}=(I-\mathbf{1}\sigma^\top)\overline{d\theta}$ there; we use them again in the fifth line to reverse it. In the fourth line, we isolate $J(\sigma)$ and use $k\sigma\mathbf{1}^\top(\sigma-\hat{\sigma})=0$ in it. In the last line, we use $r=F(\sigma)+v\mathbf{1}$ and $\hat{\sigma}^\top F(\sigma)+v=\mathbf{1}^\top\mu$.
\begin{align*}
    & d\left(\sigma^\top r-\mu^\top\ln\sigma-\mu^\top \ln r\right)=\left(\sigma-\mu/r\right)^\top dr+\left(r-\mu/\sigma\right)^\top d\sigma                                                                                       \\
  = & \left(\sigma-\mu/r\right)^\top \left(\frac{\partial F(\sigma)}{\partial\sigma}\circ\sigma\circ\frac{d\sigma}{\sigma}+dv\mathbf{1}\right) +\left(\sigma-\mu/r\right)^\top\left(r\circ\frac{d\sigma}{\sigma}\right)           \\
  = & \left(\sigma-\hat{\sigma}\right)^\top\left(I-\mathbf{1}\sigma^\top\right)\left(\frac{\partial F(\sigma)}{\partial\sigma}\circ\sigma+r\circ I\right)\left(I-\mathbf{1}\sigma^\top\right)\overline{d\theta}                   \\
  = & \left(\sigma-\hat{\sigma}\right)^\top\left(J(\sigma)+\left(I-\mathbf{1}\sigma^\top\right)\left(\left(r-(I-\mathbf{1}\sigma^\top)F(\sigma)\right)\circ I\right)\left(I-\mathbf{1}\sigma^\top\right)\right)\overline{d\theta} \\
  = & \left(\sigma-\hat{\sigma}\right)^\top \left(J(\sigma)+\left(r-(I-\mathbf{1}\sigma^\top)F(\sigma)\right)\circ I\right)\overline{d\theta}                                                                                     \\
  = & \left(\sigma-\hat{\sigma}\right)^\top \left(J(\sigma)+\left(\mathbf{1}^\top\mu+(\sigma-\hat{\sigma})^\top F(\sigma)\right) I\right)\overline{d\theta}
\end{align*}
Ignoring the higher-order term $((\sigma-\hat{\sigma})^\top F(\sigma))(\sigma-\hat{\sigma})$, we obtain the inexact gradient of MCP~\eqref{equ_mcp} in equation~\eqref{equ_grad}.
\begin{equation}
  \label{equ_grad}
  \nabla\left(\sigma^\top r-\mu^\top\ln\sigma-\mu^\top \ln r\right)=\left(J(\sigma)+(\mathbf{1}^\top\mu) I\right)^\top\left(\sigma-\hat{\sigma}\right)
\end{equation}

Next we derive a regularized Newton equation from the standard Newton equation~\eqref{equ_newton}, with the gradient $J_G^\top(\sigma-\hat{\sigma})$ on the right-hand side. The derivation is as follows. In the first line, denote $v_g=\sigma^\top F(\sigma)-\mathbf{1}^\top\mu$ such that $\sigma\circ F(\sigma)+v_g\sigma-\mu=G(\sigma,\mu)$. In the second line, denote $(\hat{\sigma},r,v_b)=M(\sigma,\mu)$ such that $\hat{\sigma}\circ(F(\sigma)+v_b\mathbf{1})=\mu$. In the third line, denote $v_{gb}=(v_g-v_b)/(\mathbf{1}^\top\mu+k)$ and use $J(\sigma)\mathbf{1}=k\mathbf{1}$.
\begin{align*}
  \sigma\circ\left(J(\sigma)+(\mathbf{1}^\top\mu)I\right)\overline{d\theta}                                         & =-\left(\sigma\circ F(\sigma)+v_g\sigma-\mu\right) \\
  \sigma\circ\left(J(\sigma)+(\mathbf{1}^\top\mu)I\right)\overline{d\theta}+(v_g-v_b)\sigma                         & =-\left(\sigma\circ F(\sigma)+v_b\sigma-\mu\right) \\
  \sigma\circ\left(J(\sigma)+(\mathbf{1}^\top\mu)I\right)\left(\overline{d\theta}+v_{gb}\mathbf{1}\right)           & =-\left(\sigma\circ F(\sigma)+v_b\sigma-\mu\right) \\
  \frac{\sigma}{r}\circ\left(J(\sigma)+(\mathbf{1}^\top\mu)I\right)\left(\overline{d\theta}+v_{gb}\mathbf{1}\right) & =-\left(\sigma-\hat{\sigma}\right)
\end{align*}

This yields the regularized Newton equation~\eqref{equ_regu_barr}. Similar to regularized Newton equation~\eqref{equ_regu_kkt}, regularized Newton equation~\eqref{equ_regu_barr} mixes Newton and gradient steps: it reduces to the standard Newton step~\eqref{equ_newton} if $\delta_H=0$, and approximates the gradient step $-(1/\delta_H)J_G^\top(\sigma-\hat{\sigma})$ if $\delta_H$ is large.
\begin{equation}
  \label{equ_regu_barr}
  \begin{aligned}
    \left(J_G^\top\circ\frac{\sigma}{r}\circ J_G+\delta_H I\right)d\theta= & -J_G^\top\left(\sigma-\hat{\sigma}\right)   \\
    \overline{d\theta}=                                                    & \left(I-\mathbf{1}\sigma^\top\right)d\theta \\
  \end{aligned}
\end{equation}

Similar to regularized Newton equation~\eqref{equ_regu_kkt}, regularized Newton equation~\eqref{equ_regu_barr} also corresponds to a least squares problem~\eqref{equ_square_barr} with a local norm. It can be verified that $J_G^\top(\sigma-\hat{\sigma})$ is the inexact gradient of $(\sigma-\mu/r)^\top(r-\mu/\sigma)$ by a derivation similar to that for the barrier MCP~\eqref{equ_mcp}.
\begin{equation}
  \label{equ_square_barr}
  \min_\sigma\left\lVert \sigma\circ r-\mu\right\rVert_{{\rm diag}(1/(\sigma\circ r))}=\min_\sigma\left(\sigma-\mu/r\right)^\top\left(r-\mu/\sigma\right)
\end{equation}

\subsection{The global linear convergence proof}

We analyze the convergence rate of the predictor-corrector path-following framework with the standard Newton corrector~\eqref{equ_newton}. First, we bound the function $\tilde{G}(\sigma,\mu)$ and its derivatives on the neighborhood $\mathcal{N}_{\tilde{\Gamma}}(\zeta)$. The function $\tilde{G}(\sigma,\mu)$ is bounded on the neighborhood $\mathcal{N}_{\tilde{\Gamma}}(\zeta)$, because $I-\mathbf{1}\sigma^\top$ and $F(\sigma)$ are bounded on the compact set $p_\Delta(\mathcal{N}_{\tilde{\Gamma}}(\zeta))$ by continuity, and $\mu/\sigma$ is bounded on $\mathcal{N}_{\tilde{\Gamma}}(\zeta)$ by Lemma~\ref{thm_bounds}. The first derivative of $\tilde{G}(\sigma,\mu)$ is derived as follows.
\begin{align*}
  d\tilde{G}(\sigma,\mu)= & d\frac{G(\sigma,\mu)}{\sigma}=\frac{1}{\sigma}\circ\frac{\partial G(\sigma,\mu)}{\partial\sigma}\circ\sigma\circ\overline{d\theta}-G(\sigma,\mu)\circ(1/\sigma^2)\circ\sigma\circ\overline{d\theta} \\
  =                       & \left(J(\sigma)+(\mathbf{1}^\top\mu)I\right)\overline{d\theta}-\tilde{G}(\sigma,\mu)\circ\overline{d\theta}
\end{align*}
$\tilde{G}'(\sigma,\mu)$ is bounded on $\mathcal{N}_{\tilde{\Gamma}}(\zeta)$, because $J(\sigma)$ is bounded on the compact set $p_\Delta(\mathcal{N}_{\tilde{\Gamma}}(\zeta))$ by continuity, and $\mathbf{1}^\top\mu$ and $\tilde{G}(\sigma,\mu)$ are bounded on $\mathcal{N}_{\tilde{\Gamma}}(\zeta)$. The second derivative satisfies the following bound.
\begin{align*}
  \left\lVert\tilde{G}''(\sigma,\mu)\right\rVert & =\left\lVert J'(\sigma)-\left(\tilde{G}(\sigma,\mu)\circ I\right)'\right\rVert\leq\left\lVert J'(\sigma)\right\rVert+\left\lVert\tilde{G}'(\sigma,\mu)\right\rVert                  \\
                                                 & \leq                                           \left\lVert J'(\sigma)\right\rVert+\left\lVert J(\sigma)\right\rVert+\mathbf{1}^\top\mu+\left\lVert\tilde{G}(\sigma,\mu)\right\rVert
\end{align*}
$\tilde{G}'(\sigma,\mu)$ is bounded on $\mathcal{N}_{\tilde{\Gamma}}(\zeta)$, because $J'(\sigma)$ and $J(\sigma)$ are bounded on the compact set $p_\Delta(\mathcal{N}_{\tilde{\Gamma}}(\zeta))$ by continuity, and $\mathbf{1}^\top\mu$ and $\tilde{G}(\sigma,\mu)$ are bounded on $\mathcal{N}_{\tilde{\Gamma}}(\zeta)$.

To prove the overall convergence rate, we prove the convergence rate of the corrector subiteration and show that each predictor step does not leave the corrector convergence region. This involves proving a uniform bound on the contraction of $\tilde{G}(\sigma,\mu)$ under corrector subiterations and a uniform bound on the change in $\tilde{G}(\sigma,\mu)$ caused by predictor steps, which leads to the following two propositions.

Adapting the standard quadratic convergence proof of the Newton iteration \cite{quadratic_proof}, we obtain the following proposition, which states that $\tilde{G}(\sigma,\mu)$ achieves uniform quadratic convergence to $0$ under the corrector.
\begin{proposition}
  \label{thm_corrector}
  Let $\theta_{k+1}=\theta_k+\overline{d\theta}_k$, where $\overline{d\theta}_k$ is given by Newton equation~\eqref{equ_newton}. Then for any $(\sigma_k,\mu)\in\mathcal{N}_{\tilde{\Gamma}}(\zeta)$, there exists $C_1>0$ such that
  \begin{equation}
    \left\lVert\tilde{G}(\sigma_{k+1},\mu)\right\rVert\leq C_1\left\lVert\tilde{G}(\sigma_k,\mu)\right\rVert ^2.
  \end{equation}
\end{proposition}
\begin{proof}
  We have the following derivation. In the third and fourth lines, we use the fundamental theorem of calculus to substitute the residuals, where $\sigma_u={\rm softmax}(\theta_k+u(\theta_{k+1}-\theta_k))$ and $\sigma_w={\rm softmax}(\theta_u+w(\theta_u-\theta_k))$. In the fifth line, we replace $\theta_{k+1}-\theta_k$ with $\overline{d\theta}_k$. The constant $L_1$ in the last line follows from the bound on the second derivative $\tilde{G}''(\sigma,\mu)$.
  \begin{align*}
         & \left\lVert\tilde{G}(\sigma_{k+1},\mu)\right\rVert=\left\lVert\tilde{G}(\sigma_{k+1},\mu)-\tilde{G}(\sigma_k,\mu)-\left(J(\sigma_k)+(\mathbf{1}^\top\mu) I\right)(\theta_{k+1}-\theta_k)\right\rVert                                  \\
    =    & \left\lVert\tilde{G}(\sigma_{k+1},\mu)-\tilde{G}(\sigma_k,\mu)-\left(\tilde{G}'(\sigma_k,\mu)+\tilde{G}(\sigma_k,\mu)\circ I\right)(\theta_{k+1}-\theta_k)\right\rVert                                                                \\
    \leq & \int_{0}^{1}\left\lVert\tilde{G}'(\sigma_u,\mu)-\tilde{G}'(\sigma_k,\mu)\right\rVert \left\lVert\theta_{k+1}-\theta_k\right\rVert \,du+\left\lVert\tilde{G}(\sigma_k,\mu)\right\rVert\left\lVert\theta_{k+1}-\theta_k\right\rVert     \\
    \leq & \int_{0}^{1}\left\lVert\tilde{G}''(\sigma_w,\mu)\right\rVert u\left\lVert\theta_{k+1}-\theta_k\right\rVert ^2\,du+\left\lVert\tilde{G}(\sigma_k,\mu)\right\rVert\left\lVert\theta_{k+1}-\theta_k\right\rVert                          \\
    \leq & \left(\frac{1}{2}L_1\left\lVert\left(J(\sigma_k)+(\mathbf{1}^\top\mu) I\right)^{-1}\right\rVert+1\right) \left\lVert\left(J(\sigma_k)+(\mathbf{1}^\top\mu) I\right)^{-1}\right\rVert\left\lVert\tilde{G}(\sigma_k,\mu)\right\rVert ^2
  \end{align*}

  In the above, $\lVert(J(\sigma)+(\mathbf{1}^\top\mu) I)^{-1}\rVert$ is uniformly bounded by $1/s_{\min}(J_G)$ on $\mathcal{N}_{\tilde{\Gamma}}(\zeta)$, and the constant $L_1$ is uniformly bounded by the bound on $\tilde{G}''(\sigma,\mu)$ on $\mathcal{N}_{\tilde{\Gamma}}(\zeta)$. This proves the inequality.
\end{proof}

For the predictor step along the path $\tilde{\Gamma}$, $\mu$ is updated along linear segments such as $\mu=(1-\gamma)\hat{\mu}_{q-1}+\gamma\mu_q$. A natural update is $-\eta u_t$ with $\eta\in(0,1)$ and a direction vector $u_t$; for a point $\mu_t$ on some segment, we take $u_t$ given by
\begin{equation}
  u_t:=\begin{cases}
    \mu_t-\mu_q,           & \text{if } \mu_q=0    \\
    \hat{\mu}_{q-1}-\mu_q, & \text{if } \mu_q\neq0
  \end{cases}.
\end{equation}
This setting has two important properties: (1) finitely many iterations suffice to reach the last segment of $\tilde{\Gamma}$, because $u_t$ has constant length except on the last segment where $\mu_q=0$; (2) $u_t/\mu_t$ is bounded along $\tilde{\Gamma}$, because $u_t=\mu_t$ when $\mu_t$ is on the last segment, where $\mu_t\to0$, and $u_t=\hat{\mu}_{q-1}-\mu_q$ has only finitely many choices on the segments where $\mu_t$ does not approach $0$.

Then we have the following proposition, which states that the change in $\tilde{G}(\sigma,\mu)$ caused by the predictor is uniformly bounded with the step size $\eta$ as coefficient.
\begin{proposition}
  \label{thm_predictor}
  Let $\mu_{t+1}=\mu_t-\eta u_t$ and $\theta_{t+1}=\theta_t+\overline{d\theta}_t$, where $\eta\in(0,1)$, and $\overline{d\theta}_t$ is given by differential equation~\eqref{equ_diff} with $d\mu_t=-\eta u_t$. Then for any $(\sigma_t,\mu_t)\in\mathcal{N}_{\tilde{\Gamma}}(\zeta)$, there exists $C_2>0$ such that
  \begin{equation}
    \left\lVert\tilde{G}(\sigma_{t+1},\mu_{t+1})-\tilde{G}(\sigma_t,\mu_t)\right\rVert\leq C_2\eta.
  \end{equation}
\end{proposition}
\begin{proof}
  We have the following derivation. In the second line, we use the fundamental theorem of calculus to substitute the residual, where $\sigma_u={\rm softmax}(\theta_t+u(\theta_{t+1}-\theta_t))$. In the third line, we replace $\theta_{t+1}-\theta_t$ with $\overline{d\theta}_t$. The constant $L_2$ in the last line follows from the bound on the first derivative $\tilde{G}'(\sigma,\mu)$.
  \begin{align*}
         & \left\lVert\tilde{G}(\sigma_{t+1},\mu_{t+1})-\tilde{G}(\sigma_t,\mu_t)\right\rVert\leq\left\lVert\tilde{G}(\sigma_{t+1},\mu_{t+1})-\tilde{G}(\sigma_t,\mu_{t+1})\right\rVert+\left\lVert\tilde{G}(\sigma_t,\mu_{t+1})-\tilde{G}(\sigma_t,\mu_t)\right\rVert \\
    \leq & \int_{0}^{1}\left\lVert\tilde{G}'(\sigma_u,\mu_{t+1})\right\rVert\left\lVert\theta_{t+1}-\theta_t\right\rVert \,du+\left\lVert\left(I-\mathbf{1}\sigma_t^\top\right)(d\mu_t/\sigma_t)\right\rVert                                                           \\
    \leq & \left(L_2\left\lVert\left(J(\sigma_t)+(\mathbf{1}^\top\mu_t)I\right)^{-1}\right\rVert+1\right) \left\lVert I-\mathbf{1}\sigma_t^\top\right\rVert\eta\left\lVert u_t/\mu_t\right\rVert \left\lVert \mu_t/\sigma_t\right\rVert
  \end{align*}

  In the above, $\lVert(J(\sigma)+(\mathbf{1}^\top\mu) I)^{-1}\rVert$ is uniformly bounded by $1/s_{\min}(J_G)$ on $\mathcal{N}_{\tilde{\Gamma}}(\zeta)$, and the constant $L_2$ is uniformly bounded by the bound on $\tilde{G}'(\sigma,\mu)$ on $\mathcal{N}_{\tilde{\Gamma}}(\zeta)$. The term $\lVert I-\mathbf{1}\sigma^\top\rVert$ is bounded on the compact set $p_\Delta(\mathcal{N}_{\tilde{\Gamma}}(\zeta))$ by continuity. The term $\lVert u_t/\mu_t\rVert$ is bounded by our setting of $u_t$, and $\lVert\mu/\sigma\rVert$ is bounded on $\mathcal{N}_{\tilde{\Gamma}}(\zeta)$ by Lemma~\ref{thm_bounds}. This proves the inequality.
\end{proof}

By Proposition~\ref{thm_corrector}, $\tilde{G}(\sigma,\mu)$ achieves uniform quadratic convergence to $0$ under the inner iteration. By Proposition~\ref{thm_predictor}, there exists a uniform $\eta\in(0,1)$ such that $\tilde{G}(\sigma,\mu)$ remains in the convergence region of the inner iteration after each predictor step. Under the outer iteration, $\mu$ achieves linear convergence to $0$ along the last segment of the path $\tilde{\Gamma}$, where the remaining segments do not change the convergence rate since there are finitely many of them. Thus, we obtain the convergence rate of the overall predictor-corrector iteration.
\begin{proposition}
  Let the predictor-corrector iteration be given as in Propositions~\ref{thm_corrector} and~\ref{thm_predictor}. For each $t$, let the Newton corrector perform sufficiently many iterations, and denote the iteration result by $\sigma_t$. Then, on the neighborhood $\mathcal{N}_{\tilde{\Gamma}}(\zeta)$, there exist $C>0$ and $\eta\in(0,1)$ such that
  $${\rm gap}(\sigma_t)<C(1-\eta)^t$$
  holds until ${\rm gap}(\sigma_t)\leq\zeta$. In particular, it is sufficient that each Newton corrector step take $O(\log t)$ subiterations.
\end{proposition}
\begin{proof}
  We first study the Newton corrector subiteration in Proposition~\ref{thm_corrector}. By the contraction inequality $\lVert\tilde{G}(\sigma_{k+1},\mu)\rVert\leq (C_1\lVert\tilde{G}(\sigma_k,\mu)\rVert)\lVert\tilde{G}(\sigma_k,\mu)\rVert$, the convergence region is $\lVert\tilde{G}(\sigma_0,\mu)\rVert<1/C_1$. By induction, the error bound is $C_1\lVert\tilde{G}(\sigma_k,\mu)\rVert\leq (C_1\lVert\tilde{G}(\sigma_0,\mu)\rVert) ^{2^k}$. Thus, the number of iterations needed to achieve $\epsilon_G$ is $$O\left(\log_2\frac{\ln(C_1\epsilon_G)}{\ln(C_1\lVert\tilde{G}(\sigma_0,\mu)\rVert)}\right)=O\left(\log\log\frac{1}{\epsilon_G}\right).$$

  Then we study the predictor step in Proposition~\ref{thm_predictor}. By
  $$\left\lVert\tilde{G}(\sigma_{t+1},\mu_{t+1})\right\rVert\leq\left\lVert\tilde{G}(\sigma_t,\mu_t)\right\rVert+\left\lVert\tilde{G}(\sigma_{t+1},\mu_{t+1})-\tilde{G}(\sigma_t,\mu_t)\right\rVert\leq\epsilon_G+C_2\eta,$$
  we have that, for predictor steps with $\eta<(1/C_1-\epsilon_G)/C_2$, $\tilde{G}(\sigma,\mu)$ remains in the convergence region of the Newton corrector subiteration.

  Then we study the iteration of $\mu_t$. By the setting $u_t=\mu_t-\mu_q$ when $\mu_q=0$ and $u_t=\hat{\mu}_{q-1}-\mu_q$ otherwise, finitely many iteration steps suffice to reach the last segment where $\mu_q=0$. Let $\mu_T$ lie on the last segment; then $\mu_t=\mu_T(1-\eta)^{t-T}$ for $t\geq T$ by the predictor steps. Denote
  $$C^-_i=\min_{0\leq t\leq T}\frac{\mu_{t,i}}{(1-\eta)^t},~C^+_i=\max_{0\leq t\leq T}\frac{\mu_{t,i}}{(1-\eta)^t},$$
  then $C^-_i(1-\eta)^t\leq\mu_{t,i}\leq C^+_i(1-\eta)^t$ for every $t$. It follows that $\min_i\mu_{t,i}\geq(\min_i C^-_i)(1-\eta)^t$ and $\mathbf{1}^\top\mu_t\leq(\sum_iC^+_i)(1-\eta)^t$.

  Finally, by ${\rm gap}(\sigma)=\mathbf{1}^\top\mu-\min_i((G(\sigma,\mu)+\mu)/\sigma)_i$ in Theorem~\ref{thm_section}, if $\lVert G(\sigma_k,\mu_t)\rVert<\min_i\mu_{t,i}$, then ${\rm gap}(\sigma_k)<\mathbf{1}^\top\mu_t\leq(\sum_iC^+_i)(1-\eta)^t$, where a sufficient condition for the former is $\lVert\tilde{G}(\sigma_k,\mu_t)\rVert<(\min_i C^-_i)(1-\eta)^t$.

  Note that Propositions~\ref{thm_corrector} and~\ref{thm_predictor} are proved on the neighborhood $\mathcal{N}_{\tilde{\Gamma}}(\zeta)$, where the path $\tilde{\Gamma}$ is truncated at $\mathbf{1}^\top\mu=\zeta$. This implies that $\mathbf{1}^\top\mu_t\geq\zeta$, so ${\rm gap}(\sigma_k)<\mathbf{1}^\top\mu_t\leq(\sum_iC^+_i)(1-\eta)^t$ holds until ${\rm gap}(\sigma_t)<\zeta$.

  Therefore, if for every outer iteration index $t$ the inner iteration iterates $\sigma_k$ to
  $$\lVert\tilde{G}(\sigma_k,\mu_t)\rVert<\min\left(1/C_1,(\min_i C^-_i)(1-\eta)^t\right),$$
  then, with the inner iteration result $\sigma_t:=\sigma_k$, the outer iteration satisfies ${\rm gap}(\sigma_t)<(\sum_iC^+_i)(1-\eta)^t$. Then, substituting $\epsilon_G=\min(1/C_1,(\min_i C^-_i)(1-\eta)^t)$, the number of subiterations needed is $O(\log t)$ for sufficiently large $t$.
\end{proof}

In conclusion, as the predictor-corrector path-following algorithm with Newton corrector solves ${\rm VI}(F,\Delta)$, ${\rm gap}(\sigma)$ achieves global linear convergence up to the prescribed precision $\zeta$ on the neighborhood $\mathcal{N}_{\tilde{\Gamma}}(\zeta)$, where $\zeta=0$ when the target solution is nonsingular and $\zeta>0$ when it is singular.

When the algorithm is applied to the general ${\rm VI}(H,K)$ with continuous mapping and polytope domain via ${\rm VI}(F_\epsilon,\Delta)$ after preprocessing, the gap function satisfies the following equation, where the bound on ${\rm gap}_H(V\sigma)-{\rm gap}_{F_\epsilon}(\sigma)$ is from equation~\eqref{equ_gap_error} with $\epsilon_{\rm set}=0$ and $\epsilon_{\rm map}=\epsilon$.
\begin{equation}
  {\rm gap}_H(V\sigma)\leq{\rm gap}_{F_\epsilon}(\sigma)+{\rm gap}_H(V\sigma)-{\rm gap}_{F_\epsilon}(\sigma)\leq C(1-\eta)^t+2\epsilon
\end{equation}

Therefore, applying the algorithm to ${\rm VI}(H,K)$, the gap function also achieves global linear convergence up to a prescribed precision $\zeta$. In particular, it can further be shown that if $H$ is continuously differentiable at the target solution, we can have $\zeta=0$.

\section{Numerical experiments}

We test the algorithm on randomly generated instances of ${\rm VI}(F,\Delta)$, with $F$ represented by tanh-activated neural networks. We evaluate two neural network architectures: (i) architecture $[n,n,n]$ with dimension $n$ ranging from $2$ to $1024$; (ii) architecture $[256,\dots,256]$ with depth ranging from $2$ to $8$. The parameters of the neural networks are generated by Xavier normal initialization, which prevents the tanh activations from being saturated and hence decreases the complexity of $F$.

The starting point $\sigma_{\rm init}$ is also randomly generated, and the required precision is $\epsilon=10^{-5}$. For each data point, we run 200 instances, totaling 6400 instances. Table~\ref{table} reports the results for two variants of the regularized Newton corrector, equation~\eqref{equ_regu_kkt} and equation~\eqref{equ_regu_barr}.

\begin{table}[htbp]
  \centering
  \caption{Statistics (mean / median) solving ${\rm VI}(F,\Delta)$}
  \label{table}
  \begin{tabular}{lcccc}
    \toprule
    \multirow{2}{*}{}
                    & \multicolumn{2}{c}{Regularized Newton corrector~\eqref{equ_regu_kkt}}
                    & \multicolumn{2}{c}{Regularized Newton corrector~\eqref{equ_regu_barr}}                                                                                                       \\
    \cmidrule(lr){2-3} \cmidrule(lr){4-5}
    Dimension       & Iteration count                                                        & \makecell{Singularity\\avoidance times} & Iteration count & \makecell{Singularity\\avoidance times} \\
    \midrule
    $[2]\times3$    & 280 / 291                                                              & 0.1 / 0.0                               & 294 / 309       & 0.1 / 0.0                               \\
    $[4]\times3$    & 353 / 327                                                              & 1.2 / 0.0                               & 360 / 339       & 1.3 / 0.0                               \\
    $[8]\times3$    & 386 / 334                                                              & 2.8 / 0.0                               & 377 / 338       & 2.0 / 0.0                               \\
    $[16]\times3$   & 436 / 341                                                              & 4.5 / 0.0                               & 410 / 332       & 4.1 / 0.0                               \\
    $[32]\times3$   & 554 / 359                                                              & 10.7 / 0.0                              & 511 / 335       & 10.8 / 0.0                              \\
    $[64]\times3$   & 804 / 521                                                              & 24.8 / 8.0                              & 723 / 501       & 24.8 / 8.5                              \\
    $[128]\times3$  & 1194 / 726                                                             & 46.2 / 20.0                             & 1009 / 656      & 43.6 / 20.5                             \\
    $[256]\times3$  & 1062 / 648                                                             & 40.4 / 18.5                             & 954 / 598       & 41.1 / 17.0                             \\
    $[512]\times3$  & 1062 / 634                                                             & 40.3 / 17.5                             & 927 / 579       & 39.8 / 17.5                             \\
    $[1024]\times3$ & 948 / 553                                                              & 35.0 / 10.0                             & 892 / 452       & 38.3 / 8.5                              \\
    \midrule
    $[256]\times2$  & 1650 / 981                                                             & 72.9 / 34.5                             & 1258 / 837      & 60.9 / 33.5                             \\
    $[256]\times3$  & 1062 / 648                                                             & 40.4 / 18.5                             & 954 / 598       & 41.1 / 17.0                             \\
    $[256]\times4$  & 902 / 571                                                              & 31.1 / 12.0                             & 870 / 512       & 35.5 / 11.5                             \\
    $[256]\times5$  & 823 / 543                                                              & 26.7 / 9.0                              & 717 / 474       & 25.4 / 8.5                              \\
    $[256]\times6$  & 685 / 442                                                              & 19.1 / 5.0                              & 605 / 406       & 18.2 / 4.5                              \\
    $[256]\times7$  & 757 / 424                                                              & 23.8 / 4.5                              & 621 / 396       & 19.7 / 3.5                              \\
    $[256]\times8$  & 717 / 422                                                              & 21.0 / 3.5                              & 584 / 376       & 16.8 / 3.0                              \\
    \bottomrule
  \end{tabular}
\end{table}

The algorithm succeeds on every one of the 6400 instances, demonstrating the robustness and scalability of the framework. The variant based on the regularized Newton corrector~\eqref{equ_regu_kkt} is slightly slower, because the term $\tilde{G}(\sigma,\mu)$ on its right-hand side involves dividing by $\sigma$ and may cause overflow when $\sigma$ is near the boundary. For architecture $[n,n,n]$, the iteration count first increases and then decreases as the dimension $n$ increases. This is probably because, although the complexity of $F$ increases with the dimension $n$, the number of solutions also increases, so the path length to a solution does not necessarily increase. For architecture $[256]\times d$, the iteration count decreases as the neural network depth $d$ increases. This stems from Xavier normal initialization. Specifically, for tanh-activated neural networks with Xavier normal initialization, the function complexity actually decreases as the depth increases.

\section{Conclusion}

This paper has introduced a path-following framework for finite-dimensional VIs with general continuous functions and compact convex domains. The framework directly handles smooth VIs on simplices after a preprocessing.

The central innovation is the fixed-point bundle, a fiber-bundle structure that geometrizes the solution set of the fixed-point KKT system. This system characterizes paths that solve the smooth VIs on simplices. The geometry separates motion along the base from motion along fibers in path-following, unifying starting point selection, path-following, and singularity avoidance in a systematic way.

The path-following algorithm requires no assumptions. It achieves global linear convergence to nonsingular solutions and retains global linear convergence up to any prescribed precision for singular solutions. Numerical experiments on 6400 randomly generated problems up to dimension 1024 confirm universal success.

The current framework uses an inner simplex approximation and analytic smoothing. Although theoretically controllable, tighter integration of these preprocessing steps with the path-following phase could further improve practical performance.

\bibliographystyle{unsrt}

\end{document}